\documentclass[11pt]{amsart}
\usepackage{amssymb,amsfonts}
\usepackage[all]{xy}
\usepackage{amsmath,amsthm,amsfonts,amssymb}
\usepackage[usenames]{color}
\usepackage{tikz}
\usepackage[all]{xy}
\usetikzlibrary{arrows,automata,positioning,matrix}

\tikzstyle{V}=[fill=black,circle,scale=0.4, outer sep = 4pt]

\usepackage{hyperref}
\usepackage{soul}
\usepackage{graphicx}
\usepackage{amssymb}
\xyoption{graph}

\numberwithin{equation}{section}

\newtheorem{theorem}{Theorem}[section]
\newtheorem{definition}[theorem]{Definition}
\theoremstyle{definition}
\newtheorem{example}[theorem]{Example}
\newtheorem{lemma}[theorem]{Lemma}
\theoremstyle{lemma}
\newtheorem{corollary}[theorem]{Corollary}
\theoremstyle{corollary}
\newtheorem{remark}[theorem]{Remark}

\newtheorem{proposition}[theorem]{Proposition}
\theoremstyle{proposition}

\newcommand{\be}{\begin{enumerate}}
\newcommand{\ee}{\end{enumerate}}
\newcommand{\beq}{\begin{equation}}
\newcommand{\eeq}{\end{equation}}

\def\N{{\mathbb{N}}}
\def\Z{{\mathbb{Z}}}

\def\bF{{\mathbb F}}

\def\bZ{{\mathbb{Z}}}

\def\C{{\mathcal{C}}}

\newcommand{\RT}{\rT} 
\DeclareMathOperator{\rT}{T}

\title{Higher dimensional digraphs from cube complexes and their spectral theory}

\author[Nadia S. Larsen]{Nadia S. Larsen}
\address[Nadia S. Larsen]{Department of Mathematics\\
University of Oslo\\
PO BOX 1053 Blindern\\
N-0316 Oslo\\
Norway}%
\email{nadiasl@math.uio.no}

\author[Alina Vdovina]{Alina Vdovina}
\address[Alina Vdovina]{Department of Mathematics\\
City College of New York\\
CUNY\\
160 Convent avenue, New York, NY, 10031 \\
USA
}
\email{avdovina@ccny.cuny.edu}

\date{3 November 2022}

\begin{document}

\maketitle

\begin{abstract}
We define $k$-dimensional digraphs and initiate a study of their spectral theory. The $k$-dimensional digraphs can be viewed as generating graphs for small categories called  $k$-graphs. Guided by geometric insight, we obtain several new series of $k$-graphs using cube complexes covered by Cartesian products of trees, for $k \geq 2$. These $k$-graphs can not be presented as virtual products, and constitute novel models of such small categories. The constructions yield rank-$k$ Cuntz-Krieger algebras for all $k\geq 2$. We introduce Ramanujan $k$-graphs satisfying optimal spectral gap property, and show explicitly how to construct the underlying $k$-digraphs.

\end{abstract}

\section{Introduction}
The study of higher rank graphs and their $C^*$-algebras originates in work of Robertson and Steger in \cite{RS} and expanded into a very active direction of research in operator algebras following the work of Kumjian and Pask \cite{KP}, where the term $k$-graph was formalised.  A higher rank graph, or $k$-graph, with $k\geq 1$, is a small category with a functor into the monoid $\N^k$ that enjoys a unique factorisation property. While many important structural results about higher rank graph $C^*$-algebras have been obtained, the supply of examples of $k$-graphs for $k\geq 3$ is limited. The main contribution of the present paper is to provide new, infinite series of examples of higher rank graphs. Our main technical innovation is the concept of a $k$-dimensional digraph, where $k\geq 2$. Using input from geometric group theory, we obtain several infinite series of explicit constructions of $k$-dimensional digraphs for $k\geq 3$, and from these we obtain novel examples of higher rank graphs with rank at least $3$ which are not of product type, in particular are not skew-products, as will be explained.

A $k$-dimensional digraph is a generalisations of a directed graph, or digraph. It is known that $1$-graphs  are free categories defined by directed graphs, see e.g. \cite{KP} or \cite[Proposition 3.12]{LSV}. Our definition of $k$-dimensional digraph is made so that the natural category associated to  it will be a $k$-graph. This definition is as follows:

\begin{definition}\label{def:k digraph from coloured graph}
Let $k\geq 2$ be a positive integer. A {$k$-dimensional digraph} $\operatorname{DG}=(V,E,o,t,k,\phi)$ is a directed graph with $V$ a finite set of vertices, $E$ finite set of edges, maps $o,t:E\to V$ which determine the origin and terminus of each $x\in E$,  and with the property that the edge set decomposes as a disjoint union $E=E_1\sqcup E_2\sqcup \dots \sqcup E_k$ with $E_i$ for $i=1,\dots, k$ regarded as edges of colour $i$, such that there is a bijection $\phi:Y\to Y$ on the set $Y:=Y_{\operatorname{DG}}$ of all directed paths of length two formed of edges of colours given by ordered pairs $(i,j)$ with $i \neq j$ in $\{1,2,\dots,k\}$, satisfying the following properties:

 \begin{enumerate}
 \item[(F1)] If $xy$ is a path of length two with $x$ of colour $i$ and $y$ of colour $j$, then $\phi (xy)=y^{\prime}x^{\prime}$ for a unique pair $(y',x')$ where $y^{\prime}$ has colour $j$, $x^{\prime}$
 has colour $i$ and the origin and terminus vertices of the paths $xy$ and
 $y^{\prime}x^{\prime}$ coincide. We write this as $xy \sim y^{\prime}x^{\prime}$ and note that $\phi^2=1$.

 \item[(F2)] For all $x\in E_i$, $y\in E_j$ and $z\in E_l$ so that $xyz$ is a path on $E$,  where $i,j,l$ are distinct colours, if  $x_1,x_2,x^2\in E_i$, $y_1,y_2,y^2\in E_j$ and $z_1,z_2,z^2\in E_l$
satisfy
\[
xy \sim y^1x^1, x^1z \sim z^1x^2,  y^1z^1 \sim z^2y^2
\]
and \[
yz \sim z_1y_1, xz_1 \sim z_2x_1, x_1y_1 \sim y_2x_2,
\] it follows that $x_2 =x^2, y_2 =y^2$ and $z_2 =z^2$.

\end{enumerate}
\end{definition}

One of the main motivations for this definition is to have a purely combinatorial finite set of data to deal with as input to defining a $C^*$-algebra. 
The definition of a $k$-graph in \cite{KP} involves an infinite category. There are combinatorial conditions describing $k$-coloured graphs with complete and associative set of squares in \cite[p. 577]{HRSW} leading to a $k$-graph,  but these too are formulated in terms of category theory and the $k$-coloured graph consists of infinite data. While the complete and associative collection of squares contains the same combinatorial information as our $k$-dimensional digraph, the advantage of our concept of $k$-dimensional digraph is two-fold, as 
it involves only a finite number of conditions to be checked without reference to graph morphisms in an infinite category, and because it allows for providing a rich supply of examples from $k$-cube complexes via geometric group theory.

If we seek for an analogy from topology, we may talk about a CW-complex having a simply-connected universal cover, so our  $k$-dimensional digraph is an analogue of the complex and the $k$-graph in the existing literature is an analogue of the universal cover.

The constructions in \cite{RS} provided generalisation of the Cuntz-Krieger algebras from topological Markov shifts introduced in \cite{CK}. They were motivated by an observation of Spielberg \cite{Spi} clarifying that a free group $\Gamma$ on finitely many generators, viewed as the fundamental group of a finite connected graph, acts on the boundary $\Omega$ of its Cayley graph such that in the associated crossed product $C(\Omega)\rtimes \Gamma$ one finds generating partial isometries  for an ordinary Cuntz-Krieger algebra $\mathcal{A}$ associated to a matrix $M$ that records incidence in the universal covering graph. This provided the basis for defining a \emph{higher rank Cuntz-Krieger} algebra $\mathcal{A}$ in \cite{RS}: the  input consists of a finite alphabet and a family of commuting $(0,1)$-matrices $M_1,M_2,\dots,M_k$, with $k\geq 1$,  having entries in the alphabet,  satisfying a number of conditions, and controlling the formation of words of $\N^k$-valued shape. Shortly after \cite{RS}, Kumjian and Pask \cite{KP} defined a  $k$-graph $\Lambda$ as a small category with a $\N^k$-valued degree of morphisms modelling the formation of words from \cite{RS} and with an associated  $C^*$-algebra $C^*(\Lambda)$  generated by partial isometries subject to Cuntz-Krieger type relations, recovering \cite{RS}.

  The explicit examples in \cite{RS} feature $k=2$.  There, the foundation for the construction of the $C^*$-algebra is a group action on the boundary of an affine $\tilde{A}_2$ building, from which one extracts a suitable alphabet and defines two commuting matrices with the properties (H0)-(H3) specified in \cite{RS}, see also \cite{KR}.  Now,  even though $k$-graphs were defined some time ago, not many explicit examples are known for $k \geq 3$, see though \cite{MRV}.

Here we construct infinite series of examples of $k$-graphs, with $k\geq 3$, from groups acting freely and  transitively on products of $k$ regular trees of constant valencies, as were explicitly found by Rungtanapirom, Stix and Vdovina in \cite{RSV}. One of our key tools is Theorem~\ref{thm:the category associated to a digraph} establishing that there is a $k$-graph, unique in the appropriate sense, associated to every $k$-dimensional digraph.

\begin{theorem}\label{thm:0}(Cf. Theorem~\ref{thm:the category associated to a digraph})
Suppose that  $\operatorname{DG}=(V,E, o,t,k,\phi)$ is a $k$-dimensional digraph, $k\geq 2$. Then there exist a small category $\Lambda_{\operatorname{DG}}$ with $V$ as set of objects and a functor $d:\Lambda_{\operatorname{DG}}\to \N^k$ which assigns value (degree) $e_i$ to morphisms determined by edges in $E_i$, for all $1\leq i\leq k$. Moreover, the functor $d$ has the unique factorisation property and in particular, $(\Lambda_{\operatorname{DG}},d)$ is a higher rank graph.
 \end{theorem}

 The other main tool shows that every $k$-cube complex gives rise to a $k$-dimensional digraph, as follows:

\begin{theorem}\label{thm:A} (Cf. Theorem~\ref{thm:one vertex k-digraph from k-cube})
For any  complex $\mathcal{X}$ with universal cover equal to the product of $k$ regular trees, where $k\geq 2$,  there is a  $k$-dimensional digraph $\operatorname{DG}(\mathcal{X})$ defined by sending a vertex of $\mathcal{X}$ to a vertex in $\operatorname{DG}(\mathcal{X})$,  and by sending  each geometric edge in $\mathcal{X}$ to two edges in the edge set of $\operatorname{DG}(\mathcal{X})$, of same colour and opposite orientations.
\end{theorem}

The $k$-graphs resulting from our constructions are not skew-products of a $k$-graph by some finite group, because such a setup would involve a fundamental group with torsion elements, while all fundamental groups of complexes in our construction are infinite torsion-free groups, since they are fundamental groups of CAT(0) spaces.

The examples of $k$-graphs we obtain by Corollary~\ref{cor:the one vertex k-graph associated to a digraph} and Theorem~\ref{thm:N-vertex k graph} differ from, for example, the $k$-graphs constructed  in \cite{KR}, \cite{R}, \cite{RS}, \cite{RS2}, since these papers only considered the case $k=2$. Moreover, our examples are new for $k=2$, see Remark~\ref{rem:our 2 graph example2.37 versus KR}. Our construction also differs from the more recent \cite{M} when $k=2$ and \cite{MRV} for $k\geq 3$, both of which emulate \cite{RS} and \cite{KR}.  The core idea  in these references is as follows: Given a  cell complex $\mathcal{X}$, each $k$-dimensional cell in $\mathcal{X}$ becomes a  vertex in a $k$-graph $\Lambda(\mathcal{X})$, and for two such cells  there is an edge in $\Lambda(\mathcal{X})$ if the given cells are adjacent via a $(k-1)$-cell, for $k\geq 2$. In this construction one may take pointed cells as vertices, or unpointed. Another  way to distinguish our higher rank graphs from the ones in \cite{RS}
is to compute the products of coordinate matrices. The products of coordinate matrices in \cite{RS} have to be $(0,1)$-matrices, but this is not necessarily the case in our construction (see Example~\ref{ex:Ramanujan from order 25}).

To place our definition in the context of similar developments, we recall that a recipe for constructing $2$-graphs was proposed in \cite[Section 6]{KP}, starting from two distinct directed graphs on the same set of vertices with commuting vertex matrices. A step forward was taken in \cite{FS}, see their Remark 2.3, where a certain associativity type condition was identified as sufficient. In \cite{HRSW}, the authors distilled these earlier attempts at constructing higher rank graphs and landed on a prescription requiring a skeleton, or a $k$-coloured  graph (where the colouring refers to edges and employs $k$ distinct colours) and a collection of building blocks termed squares  that satisfy compatibility requirements, see \cite[Theorem 4.4]{HRSW}. The squares here are certain coloured-graph morphisms.

A further simplification of the prescription of a $k$-graph from its skeleton, seen as a $k$-coloured graph, has been employed in concrete examples such as \cite[Example 7.7 ]{LLNSW} and \cite[Section 8.2]{LSV}. In fact, this last example articulates the requirements on the coloured graph that inspired our conditions (F1) and (F2) in Definition~\ref{def:k digraph from coloured graph}.
It is interesting to note that the validity of the associativity condition (F2), for $k\geq 3$, is a priori highly nontrivial. There are connections to finding solutions of the Yang-Baxter equation, see for example \cite{Y} and \cite{vdovina-YB}.

There is a strong connection between geometry of CW-complexes, groups and semigroup actions, higher rank graphs and the theory of $C^*$-algebras. The difficulty is that there are many ways to associate $C^*$-algebras to groups, semigroups and CW-complexes, and this can lead to both isomorphic and non-isomorphic $C^*$-algebras. For the higher rank graphs, there is a canonical way to associate a $C^*$-algebra, cf \cite{KP}, but it happens that non-isomorphic $k$-rank graphs lead to the same $C^*$-algebra. This conclusion is often achieved through computation of K-theory and applications of the powerful Kirchberg-Phillips classification machinery for purely infinite simple unital nuclear $C^*$-algebras.

One important question is what is a genuine higher rank? Meaning that our $k$-rank graph can not be obtained by some standard procedure from graphs of smaller ranks.
We address this question by introducing the spectral theory of combinatorial higher rank graphs.
So far the spectrum of strongly connected higher rank graphs was considered in \cite{AHLRS2} and  \cite{LLNSW},  through Perron-Frobenius theory,  which lead to new explicit constructions of von Neumann factors. We generalise results of \cite{LLNSW} by constructing infinite series of III$_{\lambda}$ factors for any $k$, and infinitely many values of $\lambda$.

We want to stress the following simple but important point about our construction of $k$-graphs: Recall that for an undirected graph with (vertex) adjacency matrix $A$, we have $A(v,w)=1=A(w,v)$ if vertices $v,w$ are connected by an edge. Thus the adjacency matrix is \emph{symmetric} and the eigenvalues are real. In a $k$-graph, we have directed edges in the various colours in its $1$-skeleton. There is no reason why the adjacency matrix should be symmetric. What can be said in general about a $k$-graph is that, if it is strongly connected, then its associated coordinate matrices jointly admit a unimodular Perron-Frobenius eigenvector, \cite{AHLRS2}. Now, in our construction of the $k$-graph $\Lambda(P)$ from a $k$-cube complex $P$, the procedure is such that it assigns to each undirected edge in the 1-skeleton of $P$ a pair of morphisms (arrows) with opposite orientation in $\Lambda(P)$. As a consequence, the adjacency matrix for the complex in direction $i$ is the same as the coordinate matrix $M_i$ of $\Lambda(P)$ in colour $i$, for all $i=1,\dots, k$. Thus all our constructions of $k$-graphs have symmetric matrices.

With this in mind, we suggest a new class of higher rank graphs, which we call Ramanujan $k$-rank graphs, see Section~\ref{sec:spectral theory}.
Their coordinate matrices are symmetric, so all eigenvalues are real and it makes sense to consider the spectral gap. We show that our $k$-graphs satisfy the optimal spectral gap condition,  which distinguishes them from the examples that have appeared in the literature so far.

The structure of the paper is the following: In a preliminary section ~\ref{sec:prelim} we collect conventions and results about categories, groups acting on products of trees, $k$-cube complexes, in particular one-vertex $k$-cube complexes from $k$-cube groups,  $k$-graphs and their $C^*$-algebras. Section~\ref{sec:one vertex k graphs}
starts with one of our main results, Theorem~\ref{thm:the category associated to a digraph}, which prescribes the construction of a $k$-graph from a given $k$-dimensional digraph. We then  associate a  $k$-dimensional digraph to any complex covered by a product of $k$ trees, see  Theorem~\ref{thm:one vertex k-digraph from k-cube}.  The ensuing Corollary~\ref{cor:the one vertex k-graph associated to a digraph} describes the  new family of $k$-graphs from $k$-cube complexes. In the case of a one-vertex complex $P$, or equivalently a $k$-cube group with complex $P$, we  show that the resulting $k$-graphs are rigid in the sense of \cite{LV}, in particular they are aperiodic and yield  classifiable $C^*$-algebras in the sense of the Kirchberg-Phillips classification \cite{Phi}.
In Section~\ref{sec:new examples} we construct $k$-cube complexes on $N$ vertices from $N$-covers of one-vertex complexes, with $N\geq 2$, and prove that the $C^*$-algebras of the associated $k$-graphs with $N$ vertices are covered by the  Kirchberg-Phillips classification theory. In Corollary~\ref{cor:von Neumann factors} we expand the scope of the constructions of $2$-graphs in \cite[Example 7.7]{LLNSW} leading to factors of type III$_{1/2}$ and give an explicit infinite family giving type III$_{{1}/{(2L)^2}}$ factors, with $L$ an arbitrary integer. In section~\ref{sec:spectral theory} we introduce the notion of Ramanujan $k$-graphs and show that there are infinite families of such $k$-graphs, see Theorem~\ref{prop: Ramanujan higher rank graphs}. Example~\ref{ex:Ramanujan from order 25}  details an explicit  Ramanujan $3$-graph on $25$ vertices with optimal spectral gap. We compute the associated $25\times 25$ incidence matrices and estimate the joint spectral gap of the $3$-dimensional digraph  with the help of MAGMA. The $3$-graph moreover features the interesting property that while it arises from an infinite $3$-cube group $\Gamma_1$ which is also an irreducible lattice, in the cover with $25$ sheets each of the three distinct alphabets in $\Gamma_1$ generates a finite group of order $25$.

{\bf Acknowledgment.}  We thank Mark Lawson and Aidan Sims for useful comments. We are grateful to the anonymous  referee for an extensive list of suggestions towards improving the presentation and several proofs.  This research was initiated during a visit by A.V. to Oslo in connection to the Master Class on "Equilibrium states in semigroup theory, K-theory and number theory", 4-6 November, 2019, supported by the Trond Mohn Foundation through the project "Pure Mathematics in Norway". It was carried out while N.L. was a Research Fellow in the Cluster of Excellence Mathematics M\"{u}nster at WWU,  Germany. She thanks for warm  hospitality and excellent working conditions in the Cluster.

\section{Preliminaries}\label{sec:prelim}
\subsection{Categories.}\label{subsection:categories}  We follow the principles laid out in \cite{MacL}, but also keep in mind the interpretation in \cite[Chapter II, \S 1.1 and \S 1.2 ]{Garside-book}. A category $\mathcal{C}$ consists of objects $\operatorname{Obj}(\mathcal{C})$ and morphisms $\operatorname{Hom}(\mathcal{C})$.  We often blur the distinction between $\mathcal{C}$ and $\operatorname{Hom}(\mathcal{C})$, and refer to the latter as the \emph{elements} or \emph{arrows} of $\mathcal{C}$. To each $f\in \mathcal{C}$ there are two objects associated, its \emph{origin}  and \emph{terminus}, and the category is seen as a collection of elements endowed with a partial product governed by compatibility of objects. Of interest to us are categories associated to directed graphs, where concatenation of edges on the graph  will determine composition of arrows in the category, upon natural reversal of origin and terminus.

\subsection{Directed graphs}
By a directed graph $D$ we mean a set $D^0$ of vertices, a set $D^1$ of edges and maps $o,t: D^1\to D^0$ determining the origin and terminus of  edges. Edges whose origin and terminus coincide, also called loops, will be allowed. We shall assume that $D^0, D^1$ are finite. We form a path $ef$ when $t(e)=o(f)$ for $e,f\in D^1$, and extend this to finite directed paths $f_1f_2\dots f_m$ on $D$, and likewise in the case of a finite number of graphs on the same vertex set, see Section~\ref{sec:one vertex k graphs}.

\vskip 0.2cm
\subsection{Complexes covered by products of trees}\label{subsec:square complex}

We start by introducing our definition of $k$-cube complex. Then we expand on the  case of one-vertex $k$-cube complexes, for which we follow  the notation and approach of \cite{vdovina-YB} and \cite{MRV}. We refer to  \cite{wise1}, \cite{burger-mozes:lattices}, and especially \cite[Section 1.2]{rattaggi:thesis} for an introduction to square complexes and $(2m,2l)$ groups, $m,l\geq 1$. We refer to \cite{sageev} for details on $CAT(0)$ complexes and to \cite{Hatcher} for the basic theory of CW complexes. We use the letter $T$ for an arbitrary regular tree, and $\mathcal{T}_l$ for the regular $l$-valent tree, where $l\geq 1$.

\begin{definition}
Let $k\geq 1$ be a positive integer. A CW complex $\mathcal{X}$ is a $k$-dimensional cube complex, or $k$-cube complex, if its universal cover is a Cartesian product of $k$ trees $\RT_1\times\RT_2\times\dots \times\RT_k$, each of which has finite constant valency.
\end{definition}

In the case of a one-vertex $k$-cube complex, for which we use the letter $P$, an equivalent definition  is as the quotient space $P=Z\backslash G$ of a group $G$ with a free and transitive action on a product $Z=\RT_1\times\RT_2\times\dots \times\RT_k$ of $k$ trees. Such $G$ are called $k$-cube groups, see Definition~\ref{defi:cubestructure}. For general $k$-cube complexes with more than one vertex, the similar definition as a quotient space $\mathcal{X}=Z\backslash G$ can be enforced upon replacing transitive action with cocompact action.

 We leave the case of trees with possibly non-constant and/or infinite valency for future discussion.

To describe a $k$-cube complex for $k\geq 2$ it is useful to recall the formalism of $2$-complexes (or square complexes) covered by products of two trees, see e.g. \cite{vdovina-YB}. We use the letter $S$ to denote a generic $2$-complex. A {\em square complex} $S$ is a $2$-dimensional combinatorial cell complex with  $1$-skeleton consisting of a graph $\mathcal{G}(S) = (V(S),E(S))$ with set of vertices $V(S)$, and set of oriented edges $E(S)$, and with $2$-cells arising from a set of squares  that are combinatorially glued to the graph $\mathcal{G}(S)$. More precisely, let $e\mapsto e^{-1}$ denote orientation reversal  of an edge  $e \in E(S)$, and suppose that $(e_1,e_2,e_3,e_4)$ is a $4$-tuple of oriented edges in $E(S)$ with the origin of $e_{i+1}$ equal to the terminus of $e_i$ (for $i$ modulo $4$). A  square $\square = (e_1,e_2,e_3,e_4)$ is the orbit of  $(e_1,e_2,e_3,e_4)$  under the dihedral action  generated by cyclically permuting the edges $e_i$ and by  the reverse orientation  map
\begin{equation}\label{eq:equivalence general square}
(e_1,e_2,e_3,e_4) \mapsto (e_4^{-1},e_3^{-1},e_2^{-1},e_1^{-1}).
\end{equation}

As customary, we think of a square-shaped $2$-cell glued to the (topological realization of the) respective edges of the graph $\mathcal{G}(S)$.

A vertical/horizontal structure (in short, a VH-structure) on a square complex is given by  a bipartite structure of the set of unoriented edges $\overline{E(S)}=E_V\sqcup E_H$ such that for every vertex $v$ in $V(S)$ the link at $v$ is the complete bipartite graph on the resulting partition $E(v)=E(v)_V\sqcup E(v)_H$, with $E(v)$ denoting the set of oriented edges with origin $v$. Torsion-free cocompact lattices $\Gamma$ in  $\mathbb Aut(\mathcal{T}_m) \times \mathbb Aut(\mathcal{T}_l)$ with $m,l\geq 1$, not interchanging the factors and considered up to conjugation, correspond uniquely to finite square complexes $S$ with a VH-structure of partition size $(2m,2l)$ up to isomorphism. Simply transitive torsion-free lattices not interchanging the factors correspond to finite square complexes with  one vertex and a VH-structure, necessarily of constant partition size.

\subsection{One-vertex $k$-cube complexes}

We first look at the case when $k=2$. Let $S$ be a square complex with  one vertex $v \in S$ and a VH-structure $\overline{E(S)} = E_V \sqcup E_H$. Pictorially, this consists of a collection of squares, each of which has four vertices labelled $v$.  Passing from the origin to the terminus of an oriented edge $e$ in a square corresponds to a fixed point free involution $e\to e^{-1}$ on $E(v)_V$ and on $E(v)_H$. Thus the partition size is necessarily a tuple $(2m,2l)$ of even integers, $m,l\geq 1$. The lattice identified with $\pi_1(S,v)$ admits a description in terms of two generating subsets $A,B$, see \cite[Definition 5]{vdovina-YB}.

\begin{definition} \label{defi:BMstructureingroup}
A  {\em vertical/horizontal structure}, or {\em VH-structure},  in a group $G$ is an ordered pair $(A,B)$ of finite subsets $A,B \subseteq G$ such that the following hold.
\begin{enumerate}
\item \label{defiitem:AandBinvolution} Taking inverses induces fixed point free involutions on $A$ and $B$.
\item The union $A \cup B$ generates $G$.
\item \label{defiitem:ABequalsBA} The product sets $AB$ and $BA$ have size $\#A \cdot \#B$ and satisfy $AB = BA$.
\item \label{defiitem:AB2torsion} The sets $AB$ and $BA$ do not contain $2$-torsion.
\end{enumerate}
The tuple $(\#A,\#B)$ is called the {\em valency vector} of the VH-structure in $G$.
\end{definition}

If a group $G$ admits a VH-structure $(A,B)$ of valency vector $(\# A, \# B)$, then by \cite[Section 6.1]{burger-mozes:lattices}, there is a square complex $S_{A,B}$ with one vertex and a VH-structure in the sense of subsection \ref{subsec:square complex}. The set of oriented edges of $S_{A,B}$ is the disjoint union $
E(S_{A,B}) = A \sqcup B$,
with the orientation reversion map given by $e \mapsto e^{-1}$, and with $A,B$ labelling the edges in vertical and horizontal direction, respectively. The link of $S_{A,B}$ in $v$ is the complete bipartite graph with vertices labelled by $A$ and $B$, see \cite[Lemma 1]{vdovina-YB}, and \cite[Theorem C]{Ballmann-Brin1995} implies that the universal cover of $S_{A,B}$ is a product of trees. Conversely, given a square complex $S$ with a VH-structure $(A,B)$ and a single vertex, its fundamental group (i.e. the fundamental group of its topological realisation) admits a VH-structure of valency $(\# A, \# B)$, see \cite[Proposition 5.7]{RSV}. We refer to it as a $(\#A, \#B)$-group. Example \ref{ex:group 37} shows a $(4,4)$-complex with associated $(4,4)$-group.

To describe the geometric squares of $S_{A,B}$, note that a relation $ab=b'a'$ in $G$ with $a,a' \in A$ and $b,b' \in B$ (not necessarily pairwise distinct), as prescribed by Definition \ref{defi:BMstructureingroup}, leads to four algebraic relations obtained  upon cyclic permutation and inversion, namely
\begin{equation}\label{eq:relationabba}
ab = b'a',\quad
a^{-1}b' =  ba'^{-1}, \quad
a'^{-1}b'^{-1}  =  b^{-1}a^{-1}\quad \text{ and }a'b^{-1}= b'^{-1}a.
\end{equation}
This leads to the definition of a geometric square as a tuple of four Euclidean squares. All four vertices in each square coalesce into the single vertex $v$ of $S_{A,B}$ when we glue the edges according to labels and orientation.  Before we introduce our convention, we recall briefly two other (equivalent) conventions for describing geometric squares.

\subsection{One convention - see e.g. Rattaggi}
The formalism of a geometric square seen as a $4$-tuple of squares in Euclidean space is well-known. In \cite[Figure 4.1, page 182]{rattaggi:thesis}, for example, the group relation $ab=b'a'$ is reflected by the $4$-tuple of squares
having edges labelled according to a one-way cyclic permutation in counterclockwise direction, see below:
\begin{equation}\label{eq:squares rattaggi}
\def\g#1{\save[]!C="g#1"\restore}%
\xymatrix{
\g1 \ar @{}[dr] |{S_O}{\bullet}\ar[d]_{b'} & {\bullet}\ar[l]_{a'} &\g2 \ar @{}[dr] |{S_R}{\bullet}\ar[d]_{b} & {\bullet}\ar[l]_{a}&\g3 \ar @{}[dr] |{S_H}{\bullet}\ar[r]^{a'} & {\bullet}\ar[d]^{b'} &\g4 \ar @{}[dr] |{S_V}{\bullet}\ar[r]^{a} & {\bullet}\ar[d]^{b}\\
{\color{blue}{{\bullet}}}\ar[r]_{a} & {\bullet}\ar[u]_{b} & {\color{blue}{{\bullet}}}\ar[r]_{a'} & {\bullet}\ar[u]_{b'} & {\color{blue}{{\bullet}}}\ar[u]^{b} & {\bullet}\ar[l]^{a} &
{\color{blue}{{\bullet}}}\ar[u]^{b'} & {\bullet}\ar[l]^{a'}
}
\end{equation}
The notation means that if $S_O$ is regarded as a reference square, then $S_H$ is obtained by reflection in the horizontal direction (about $b$), $S_V$ by reflection in the vertical direction (about $a$) and, finally, $S_R$ arises from rotation counterclockwise by $\pi$. Our use of $S_O, S_R, S_H, S_V$ as notation for the squares is inspired by \cite[Section 2]{MRV}.

The geometric square associated to $ab=b'a'$ in \cite{rattaggi:thesis} and visualised in \eqref{eq:squares rattaggi} is given by
\[
\{(a,b,a',b'),\, (a', b', a, b),\, (a^{-1},{b'}^{-1}, {a'}^{-1}, b^{-1}),\, ({a'}^{-1}, b^{-1},a^{-1},{b'}^{-1})\}.
\]

\subsection{A second convention - see Kimberley-Robertson}
In  \cite{KR},  Kimberley-Robertson adopted a two-direction labelling of their squares which to a group relation $ab=b'a'$ assigns a $4$-tuple of squares according to the convention below:
\begin{equation}\label{eq:squares kimberley robertson}
\def\g#1{\save[]!C="g#1"\restore}%
\xymatrix{
\g1 \ar @{}[dr] |{S_O}{\bullet}\ar[r]^{a'} & {\bullet} &\g2 \ar @{}[dr] |{S_R}{\bullet}\ar[d]_{b} & {\bullet}\ar[l]_{a}\ar[d]^{b'} &\g3 \ar @{}[dr] |{S_H}{\bullet} & {\bullet}\ar[l]_{a'} &\g4 \ar @{}[dr] |{S_V}{\bullet}\ar[r]^{a}\ar[d]_{b'} & {\bullet}\ar[d]^{b}\\
{\color{blue}{{\bullet}}}\ar[u]^{b'}\ar[r]_{a} & {\bullet}\ar[u]_{b} & {\color{blue}{{\bullet}}} & {\bullet}\ar[l]^{a'} & {\color{blue}{{\bullet}}}\ar[u]^{b} & {\bullet}\ar[u]_{b'}\ar[l]^{a} &
{\color{blue}{{\bullet}}}\ar[r]_{a'} & {\bullet}
}
\end{equation}
The geometric square associated to $ab=b'a'$ and visualised in \eqref{eq:squares kimberley robertson} is given by
\[
\{(a,b,b',a'),\, ({a'}^{-1}, {b'}^{-1}, b^{-1}, a^{-1}),\, (a^{-1},b', b, {a'}^{-1}), \, (a', b^{-1}, {b'}^{-1}, a)\}.
\]

\subsection{Our convention}
We choose a convention that will facilitate our constructions of $k$-graphs later on, and in particular we swap the letters for vertical and horizontal directions, as follows: we keep the cyclic permutation in counterclockwise direction from \cite{rattaggi:thesis} but choose labelling of edges as "starting" at one vertex by going out in both horizontal and vertical direction, similar to \cite{KR}.

\begin{equation}\label{eq:squares larsen vdovina}
\def\g#1{\save[]!C="g#1"\restore}%
\xymatrix{
\g1 \ar @{}[dr] |{S_O}{\bullet}\ar[r]^{b} & {\bullet} &\g2 \ar @{}[dr] |{S_R}{\bullet}\ar[d]_{a'} & {\bullet}\ar[l]_{b'}\ar[d]^{a} &\g3 \ar @{}[dr] |{S_H}{\bullet} & {\bullet}\ar[l]_{b} &\g4 \ar @{}[dr] |{S_V}{\bullet}\ar[r]^{b'}\ar[d]_{a} & {\bullet}\ar[d]^{a'}\\
{\color{blue}{{\bullet}}}\ar[u]^{a}\ar[r]_{b'} & {\bullet}\ar[u]_{a'} & {\color{blue}{{\bullet}}} & {\bullet}\ar[l]^{b} & {\color{blue}{{\bullet}}}\ar[u]^{a'} & {\bullet}\ar[u]_{a}\ar[l]^{b'} &
{\color{blue}{{\bullet}}}\ar[r]_{b} & {\bullet}
}
\end{equation}
Explicitly, we define a \emph{geometric square} as visualised in \eqref{eq:squares larsen vdovina} to be a tuple
\begin{equation}\label{eq:4 relations on a and b}
\{({a},{b},{a}'^{-1},{b}'^{-1}),\, ({a}'^{-1},{b}'^{-1},{a},{b}),\, ({a}',{b}^{-1},{a}^{-1},{b}'),\,({a}^{-1},{b}',{a}',{b}^{-1})\},
\end{equation}
where any two squares are seen as equivalent.

Since for our purposes it will be important to keep track of how such  squares arise, we introduce the following more precise notation: for $a\in A$ and $b\in B$, we let
\begin{equation}\label{eq:S ab}
S_O^{a,b}:=(a,b,a'^{-1}, b'^{-1}),
\end{equation}
where $a',b'$ are the unique elements in $A$ and $B$, respectively, such that $ab=b'a'$. We refer to $ab$ as the \emph{vertical-horizontal pair of edges in $S_{O}^{a,b}$} and to $b'a'$ as the \emph{horizontal-vertical pair of edges in $S_O^{a,b}$}.

In \cite{vdovina-YB}, the last named author generalised VH-structure to the $k$-dimensional case, as follows.

\begin{definition} (See \cite[Definition 7]{vdovina-YB}) \label{defi:cubestructure}
A {\em $k$-cube structure}  in a group $G$ is an ordered $k$-tuple $(A_1, \ldots, A_k)$ of finite subsets $A_i \subseteq G$ such that the following hold for all $i,j=1,\dots, k$, $i\neq j$:
\begin{enumerate}
\item[(1)] \label{defiitem:AandBinvolution} Taking inverses induces fixed point free involutions on $A_i$.
\item[(2)] The union $\cup A_i$ generates $G$.
\item[(3)] \label{defiitem:ABequalsBA} The product sets $A_iA_j$ and $A_jA_i$  have size $\#A_i \cdot \#A_j$ and $A_iA_j = A_jA_i$.
\item[(4)] \label{defiitem:AB2torsion} The sets $A_iA_j$ and $A_jA_i$ do not contain $2$-torsion.
\item[(5)] The group $G$ acts simply transitively on a Cartesian product of $k$ trees.
\end{enumerate}
The tuple $(\#A_1, \ldots,\#A_k)$ is the {\em valency vector} of the $k$-cube structure in $G$, and  $A_1, \ldots, A_k$ are generating sets of $G$.
\end{definition}

We note that each pair $(A_i,A_j)\subseteq G$ for $i,j=1,\dots ,k$ with $i\neq j$ forms a subgroup $G_{i,j}$ of $G$ equipped with a VH-structure. This observation can be used to show that to a given $k$-cube group $G$ with generating family $(A_1, \ldots, A_k)$ there is an associated $k$-cube complex $P_{(A_1, \ldots, A_k)}$. Its $2$-dimensional cells are prescribed by the square complexes $S_{A_i,A_j}$ obtained from each $(\#A_i,\#A_j)$ group $G_{i,j}$ for $i\neq j$.

\begin{remark}\label{rmk:k cube complkexes versus groups}
In studying $k$-cube groups and one-vertex $k$-cube complexes, we will move freely between two equivalent interpretations. Starting from a $k$-cube group $G$  defined algebraically through properties (1)-(5) in Definition~\ref{defi:cubestructure}, the associated quotient space $Z\backslash G$ with $Z$ a product of $k$ trees is a $k$-cube complex with one vertex. Its construction as a CW complex from $j$-cells for $0\leq j\leq k$ is detailed in \cite[Definition 2.3]{MRV}. Conversely, one may define a $k$-cube group $G$ from combinatorial data by starting with $k$ finite sets of even cardinalities (encoding edges), and building up by induction (on dimension of cells) a $k$-dimensional complex with the necessary compatibility to yield generating sets $A_1, A_2,\dots, A_k$ for $G$, see \cite[Definition 2.4]{MRV}.
\end{remark}

We next introduce some notation for cubes in a one-vertex $k$-cube complex $P$ for $k\geq 2$,  with motivation and inspiration drawn from \cite[Section 2]{MRV}. We let $E$ be the set of edges. Because $P$ is the complex associated to a $k$-cube group, $E$ partitions into $k$ subsets  as $E=E_1\sqcup \dots \sqcup E_k$ in such a way that if an edge $e$ is in $E_j$ then $e^{-1}\in E_j$ for $j=1,\dots, k$. We refer to $E_i$ as the subset of edges of \emph{colour} $i$, for $i=1,\dots,k$, where the $k$ colours are assumed distinct.
The $2$-cells of $P$ are  geometric squares of the form $S_O=(a,b,a'^{-1},b'^{-1})$  regarded as the equivalence class $\{S_O, S_R, S_H, S_V\}$  described in \eqref{eq:4 relations on a and b}, where $a,a'\in  E_i$ and $b,b'\in E_j$ for $i\neq j$. By a \emph{geometric square} we mean any
square in $\{S_O, S_R, S_H, S_V\}$. Similar to \cite{MRV}, for distinct colours $i\neq j$ we let
\[
F(i,j)=\{S=(a,b,a'^{-1}, b'^{-1})\mid S\text{ is a geometric square with }a,a'\in E_i, b,b'\in E_j\},
\]
and we denote by $S^{ij}$ a generic square in $F(i,j)$ for all $i\neq j$ in $\{1,\dots ,k\}$.
The $3$-cells are determined by \emph{geometric cubes}, all whose $6$ faces are geometric squares, see figure~\ref{figure generic cube} for a (generic) such cube.

\begin{figure}[h]
\centering
\includegraphics[height=4cm]{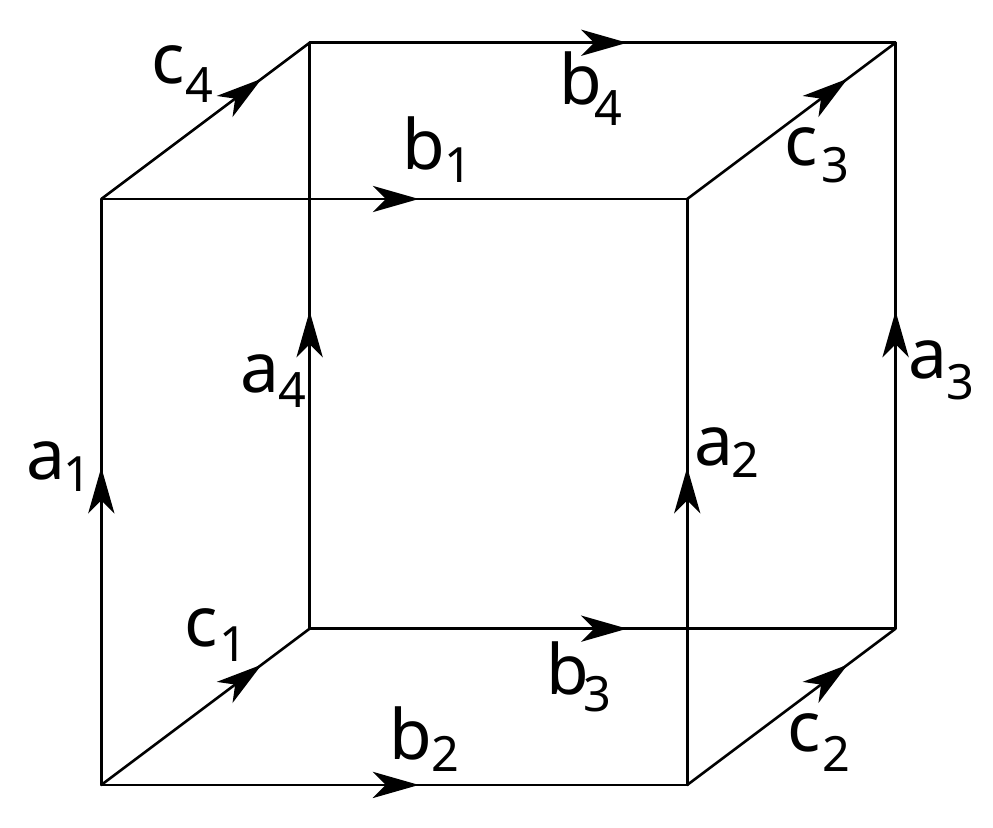}
\caption{A generic $3$-cube}\label{figure generic cube}
\end{figure}
More precisely, the $6$ faces of the cube are geometric squares $(S_1^{ij}, S_2^{il}, S_3^{jl}, S_4^{jl}, S_5^{il}, S_6^{ij})$, with
\begin{align*}
S_1^{ij}&=(a_1,b_1,a_2^{-1},b_2^{-1}),\text{ front face }\\
S_2^{il}&=(a_2,c_3,a_3^{-1},c_2^{-1}),\text{ right face}\\
S_3^{jl}&=(b_2,c_2,b_3^{-1},c_1^{-1}),\text{ bottom face}\\
S_4^{jl}&=(b_1,c_3,b_4^{-1}, c_4^{-1}),\text{ top face}\\
S_5^{il}&=(a_1,c_4,a_4^{-1},c_1^{-1}),\text{ left face}\\
S_6^{ij}&=(a_4,b_4,a_3^{-1},b_3^{-1}),\text{ back face}.
\end{align*}
In particular, any one of the $6$ geometric squares in this list is given subject to the equivalence relation \eqref{eq:4 relations on a and b}, and the geometric cube can be equivalently presented with any one of the $8$ vertices in the bottom-left position. We stress that the labels $a_1,\dots, a_4$, $b_1,\dots, b_4$, $c_1,\dots,c_4$ here are formal symbols that keep track of how the cubes are glued in the complex. As already mentioned, if $e$ is a label for an edge, then $e^{-1}$ is the label recording orientation reversal.

For $4\leq l\leq k$, the $l$-dimensional cells are $l$-cubes, see \cite{MRV}.

A given $k$-cell in a $k$-cube complex $P$ has a topological realisation as the product of intervals $[0,1]^k$. Denoting by $\varepsilon_i$ the standard basis elements in $\mathbb{R}^k$, for $i=1,\dots,k$, we view a geometric edge in $P$  as having \emph{degree $\varepsilon_i$} if it lies in the span of the generator $\varepsilon_i$  in its topological realisation.  This agrees with the degree of paths in higher rank graphs in section~\ref{sec:one vertex k graphs}.

\subsection{Examples of cube groups}\label{subsection:RSV groups}

The cube groups in these section were introduced in \cite{vdovina-YB}. They contain as a particular case arithmetic lattices and non-residually finite CAT(0) groups constructed in \cite{RSV} .
We refer to them as \emph{RSV-groups}. They are  \emph{the first explicit examples} of arithmetic groups acting freely and transitively on products of $k$ trees of constant valencies, for $k\geq 3$, as well as non-residually finite CAT(0) groups of dimensions
$k\geq 3$. RSV groups are irreducible in the sense that they can not be presented as virtual products of group  actions on products of smaller number of trees.

We recall here a construction of an explicit series of RSV lattices, which is infinite in several parameters, $k$, $q$ and $\delta$. The significance of the irreducibility of these groups is that the associated $k$-graphs can not presented as virtual products, so are entirely new.

For $q$ an odd prime, let $\delta \in \bF_{q^2}^\times$ be a generator of the multiplicative group of the field with $q^2$ elements. If $i,j \in \bZ/(q^2-1)\bZ$ satisfy $i \not\equiv j \pmod{q-1}$, then  $1+\delta^{j-i} \not= 0$, and it follows that there is a unique $x_{i,j} \in \bZ/(q^2-1)\bZ$ with $
\delta^{x_{i,j}} = 1  + \delta^{j-i}.$

Set $y_{i,j} := x_{i,j} + i - j$, and note that $
\delta^{y_{i,j}} = \delta^{x_{i,j} + i - j} = (1  + \delta^{j-i}) \cdot \delta^{i-j} = 1  + \delta^{i-j}$. Define
\begin{align*}
l(i,j) & := i -  x_{i,j}(q-1), \\
k(i,j) & := j -  y_{i,j}(q-1),
\end{align*}
and further let $M \subseteq \bZ/(q^2-1)\bZ$ be a union of cosets under $(q-1)\bZ/(q^2-1)\bZ$ with $\#M=k$.

If $q$ is odd, it was shown in \cite{RSV} that the following groups act freely and transitively on a product of $k$ trees:
\[
\Gamma_{M,\delta} = \left\langle a_i, i \in M \ \left|
\begin{array}{c}
a_{i+ (q^2 - 1)/2}  a_i = 1 \text{ for all $i \in M$}, \\
a_i a_j = a_{k(i,j)}a_{l(i,j)} \text{ for all $i,j \in M$ with $i \not\equiv j \pmod{q-1}$}
\end{array}
\right.\right\rangle
\]

\begin{example}
\label{ex:666}
The smallest example in dimension $k=3$ arises with $q=5$ and $M$ equal to the collection of cosets  $i\in\Z/{24}\Z$ with $i$ not dividing $4$. This group, denoted $\Gamma_1$, acts vertex transitively on the product of three regular trees $\mathcal{T}_{6} \times  \mathcal{T}_{6} \times  \mathcal{T}_{6}$ and has the presentation
\[
\Gamma_{1} = \left\langle \begin{array}{c}
a_1,a_5,a_9,a_{13},a_{17},a_{21}, \\
b_2,b_6,b_{10},b_{14},b_{18},b_{22}, \\
c_3,c_7,c_{11},c_{15},c_{19},c_{23}
\end{array}
\ \left|
\begin{array}{c}
a_ia_{i+12} = b_i b_{i+12} = c_ic_{i+12} = 1  \ \text{ for all $i$ }, \\

a_{1}b_{2}a_{17}b_{22}, \
a_{1}b_{6}a_{9}b_{10}, \
a_{1}b_{10}a_{9}b_{6}, \
a_{1}b_{14}a_{21}b_{14}, \
a_{1}b_{18}a_{5}b_{18}, \\
a_{1}b_{22}a_{17}b_{2}, \
a_{5}b_{2}a_{21}b_{6}, \
a_{5}b_{6}a_{21}b_{2}, \
a_{5}b_{22}a_{9}b_{22}, \\

a_{1}c_{3}a_{17}c_{3}, \
a_{1}c_{7}a_{13}c_{19}, \
a_{1}c_{11}a_{9}c_{11}, \
a_{1}c_{15}a_{1}c_{23}, \
a_{5}c_{3}a_{5}c_{19}, \\
a_{5}c_{7}a_{21}c_{7}, \
a_{5}c_{11}a_{17}c_{23}, \
a_{9}c_{3}a_{21}c_{15}, \
a_{9}c_{7}a_{9}c_{23}, \\

b_{2}c_{3}b_{18}c_{23}, \
b_{2}c_{7}b_{10}c_{11}, \
b_{2}c_{11}b_{10}c_{7}, \
b_{2}c_{15}b_{22}c_{15}, \
b_{2}c_{19}b_{6}c_{19}, \\
b_{2}c_{23}b_{18}c_{3}, \
b_{6}c_{3}b_{22}c_{7}, \
b_{6}c_{7}b_{22}c_{3}, \
b_{6}c_{23}b_{10}c_{23}.
\end{array}
\right.\right\rangle
\]

Thus $\Gamma_1$ is a $3$-cube group with $A_1=\{a_1,a_5,a_9,a_{13}, a_{17},a_{21}\}$ and similar descriptions for $A_2$ and $A_3$. It  is an arithmetic group, so it is residually finite. Of interest to us is the fact that it admits quotients of order $5^l, l\in \mathbb{N}$,
see Example~\ref{ex:Ramanujan from order 25}.
\end{example}

In \cite{RSV}, the authors also constructed $k$-cube groups acting on a product of trees of distinct constant valencies. Explicitly,  for any set of size $k$ of distinct odd primes $p_1, \ldots, p_k$, there is a group acting simply transitively on a product of $k$ trees of valencies $p_1+1,\ldots, p_k+1$, obtained using Hamiltonian quaternion algebras.

\begin{example}
For $p_1=3,p_2=5,p_3=7$, there is an explicit presentation of a group acting simply transitively on a product of three trees $
\mathcal{T}_{4} \times  \mathcal{T}_{6} \times  \mathcal{T}_{8}$,
see \cite{RSV}. Indeed, with $\mathbf{i}, \mathbf{j}, \mathbf{k}$ denoting the quaternions, let
\begin{align*}
a_1 = 1 + \mathbf{j} + \mathbf{k}, \ a_2 = 1+\mathbf{j}-\mathbf{k}, \ a_3 = 1-\mathbf{j}-\mathbf{k}, \ a_4 = 1 -\mathbf{j} + \mathbf{k}, \\
b_1 = 1 + 2\mathbf{i}, \ b_2  = 1 + 2\mathbf{j}, \ b_3 = 1 + 2\mathbf{k}, \ b_4 = 1 - 2\mathbf{i}, \ b_5 = 1- 2\mathbf{j}, \ b_6 = 1 - 2\mathbf{k}, \\
c_1 = 1+2\mathbf{i} + \mathbf{j} + \mathbf{k}, \ c_2 = 1-2\mathbf{i} + \mathbf{j} + \mathbf{k}, \ c_3 = 1+2\mathbf{i} - \mathbf{j} + \mathbf{k}, \ c_ 4= 1+2\mathbf{i} + \mathbf{j} - \mathbf{k}, \\
 c_5 = 1- 2\mathbf{i} -  \mathbf{j} - \mathbf{k}, \ c_6 = 1+2\mathbf{i} - \mathbf{j} - \mathbf{k}, \ c_7 = 1-2\mathbf{i} + \mathbf{j} - \mathbf{k}, \ c_8 = 1-2\mathbf{i} - \mathbf{j} + \mathbf{k}.
\end{align*}
Then $a_l^{-1} = a_{l+2}$, $b_l^{-1} = b_{l+3}$, and $c_l^{-1} = c_{l+4}$. The required group is given
 by
\[
\Gamma_2 = \left\langle
\begin{array}{c}
a_1,\dots a_4 \\
b_1,\dots,b_6 \\
c_1,\dots, c_8
\end{array}
\ \left|
\begin{array}{c}
a_1b_1a_4b_2,  \ a_1b_2a_4b_4, \  a_1b_3a_2b_1, \
a_1b_4a_2b_3,  \ a_1b_5a_1b_6, \ a_2b_2a_2b_6 \\

a_1c_1a_2c_8, \ a_1c_2a_4c_4, \ a_1c_3a_2c_2, \ a_1c_4a_3c_3, \\
a_1c_5a_1c_6, \ a_1c_7a_4c_1, \ a_2c_1a_4c_6, \ a_2c_4a_2c_7 \\

b_1c_1b_5c_4, \
b_1c_2b_1c_5, \
b_1c_3b_6c_1, \
b_1c_4b_3c_6, \
b_1c_6b_2c_3, \
b_1c_7b_1c_8, \\
b_2c_1b_3c_2, \
b_2c_2b_5c_5, \
b_2c_4b_5c_3, \
b_2c_7b_6c_4, \
b_3c_1b_6c_6, \
b_3c_4b_6c_3
\end{array}
\right.\right\rangle.
\]
This is also denoted $\Gamma_{3,5,7}$, see \cite{RSV}.
\end{example}

\subsection{Higher rank graphs}\label{subs: higher rank graphs} We recall the definition of a $k$-graph due to Kumjian and Pask \cite{KP}.
For an integer $k \ge 1$, we view $\N^k$ as a monoid under pointwise addition. A $k$-graph is a countable small category $\Lambda$ together with an
assignment of a \emph{degree} $d(\mu)\in\N^k$ to every morphism $\mu\in\Lambda$
such that for all $\mu,\nu,\pi\in \Lambda$ the following hold
\begin{enumerate}
\item $d(\mu\nu)=d(\mu)+d(\nu)$; and
\item whenever $d(\pi)=m+n$ for $m,n \in \N^k$, there is a unique factorisation $\pi=\mu\nu$
    such that $d(\mu)=m$ and $d(\nu)=n$.
\end{enumerate}
Condition~(2) is known  as the \emph{factorisation property} in the $k$-graph. The composition in $\mu\nu$ is understood in the sense of morphisms, thus the source $s(\mu)$ of $\mu$ equals the range $r(\nu)$ of $\nu$. Note that the morphisms of degree $0$ (in $\N^k$) are the identity morphisms in the category. Denote this set by $\Lambda^0$, and refer to its elements  as \emph{vertices} of $\Lambda$. With $e_1,\dots,e_k$ denoting the generators of $\N^k$, the set  $\Lambda^{e_i}=\{\lambda\in \Lambda\mid d(\lambda)=e_i\}$ consists of edges (or morphisms) of degree $e_i$, for  $i=1,\dots,k$. We write $v\Lambda^n$ for the set of morphisms of degree $n\in \N^k$ with range $v$.

Throughout this paper we are concerned with $k$-graphs where $\Lambda^0$ and all $\Lambda^{e_i}$, $i=1,\dots, k$, are finite. A $k$-graph $\Lambda$ so that $0<\#v\Lambda^n<\infty$ for all $v\in \Lambda^0$ and all $n\in \N^k$ is source free and row-finite. Following \cite{AHLRS2}, a finite $k$-graph $\Lambda$ is \emph{strongly connected} if $v\Lambda w \not= \emptyset$ for all vertices $v,w \in \Lambda^0$.

The \emph{coordinate matrices} $M_1, \dots, M_k\in\operatorname{Mat}_{\Lambda^0}(\N)$ of $\Lambda$
are $\Lambda^0\times \Lambda^0$ matrices with
\[
M_i(v,w) = |v\Lambda^{e_i}w|.
\]
By the factorisation property, the matrices $M_i$ pairwise commute for $i=1,\dots ,k$. For ${n}=(n_i)_{i=1,\dots,k} \in
\N^k$, we define
\begin{equation*}\label{pairwisecomm}\textstyle
M^{{n}} := \prod^k_{i=1} M_i^{n_i}.
\end{equation*}
We denote the spectral radius of a square matrix $B$ by $\rho(B)$, and we let
\[
\rho(\Lambda) := (\rho(M_1), \rho(M_2), \dots, \rho(M_k)) \in [0,\infty)^k.
\]
For ${m} \in \Z^k$ we write $\rho(\Lambda)^{{m}}$ for the product $\prod^k_{i=1}
\rho(M_i)^{m_i}$.

Given a row-finite, source free  $k$-graph $\Lambda$, its associated $C^*$-algebra $C^*(\Lambda)$ is the universal $C^*$-algebra generated by a family $\{\mathbf{s}_\mu \mid
\mu\in\Lambda\}$ of partial isometries  satisfying
\begin{itemize}
\item[(CK1)] $\{\mathbf{s}_v \mid v \in \Lambda^0\}$ is a family of mutually orthogonal
    projections;
\item[(CK2)] $\mathbf{s}_\mu \mathbf{s}_\nu = \mathbf{s}_{\mu\nu}$ whenever $s(\mu) = r(\nu)$;
\item[(CK3)] $\mathbf{s}_\mu^* \mathbf{s}_\mu = \mathbf{s}_{s(\mu)}$ for all $\mu$;
\item[(CK4)] $\mathbf{s}_v = \sum_{\mu \in v\Lambda^{{n}}} \mathbf{s}_\mu \mathbf{s}^*_\mu$ for all $v\in \Lambda^0,{n}\in \N^k$.
\end{itemize}

\section{Construction of $k$-graphs from $k$-cube groups: the one-vertex case}\label{sec:one vertex k graphs}

In this section we construct $k$-graphs with one vertex, for $k\geq 2$. There are two main steps. The first is a general  procedure by which we associate a category to a $k$-dimensional digraph as introduced in Definition~\ref{def:k digraph from coloured graph}, and prove that the conditions (F1) and (F2) inherent to the $k$-dimensional digraph imply the existence of a degree functor from the category to $\N^k$ that satisfies the required factorisation property. This step will be reminiscent of the construction of Artin monoids as quotients of free monoids by an equivalence relation on the collection of positive words identifying braid strings, see \cite{BS72}. The degree functor will be similar to the degree map on an Artin monoid, cf. \cite{Saito}. The second step is to provide $k$-dimensional digraphs, and here we shall use one-vertex cube complexes associated to $k$-cube groups as a source from which to construct such digraphs.

In Section~\ref{sec:new examples} we  use the results of this section combined with concrete constructions of covering maps in the context of complexes to produce higher rank graphs with more than one vertex. We stress that our constructions are performed on the complexes, which depend on finite combinatorial data, and not on the $k$-graphs, which are categories with additional structure.

The next result is the abstract construction of the category associated to a $k$-dimensional digraph.

\begin{theorem}\label{thm:the category associated to a digraph}
Suppose that  $\operatorname{DG}=(V,E, o,t,k,\phi)$ is a $k$-dimensional digraph, $k\geq 2$. Then there exist a small category $\Lambda_{\operatorname{DG}}$ with $V$ as set of objects and a functor $d:\Lambda_{\operatorname{DG}}\to \N^k$ which assigns value (degree) $e_i$ to morphisms determined by edges in $E_i$, for all $1\leq i\leq k$. Moreover, the functor $d$ has the unique factorisation property and in particular, $(\Lambda_{\operatorname{DG}},d)$ is a higher rank graph.
 \end{theorem}

 \begin{proof}
 First, let $\C:=\C_{\operatorname{DG}}$ be  the free category associated to the directed graph $\operatorname{DG}=(V,E,o,t)$, where we disregard the colouring of the edges, see e.g. \cite[Theorem 1, page 49]{MacL}.  The object set of $\C$ is $V$ and the arrows are given by finite strings, or paths, consisting of finite sequences $v_1,\dots,v_m$ of objects connected by  $m-1$  arrows $x_s:v_s\to v_{s+1}$ with the  compatibility of  objects $o(x_{s+1})=t(x_{s})$ for $1\leq s\leq m-1$, where $m> 1$ is arbitrary.

 To conform later with the conventions of higher rank graphs, we assign length $0$ to identity morphisms, or arrows $\langle v_1\rangle$, as opposed to length $1$  in \cite{MacL}. An arrow $\langle v_1,x_1,v_2\rangle$ of length $1$ will be called an \emph{elementary} arrow, where we again shift the value of the length by $-1$ compared to \cite{MacL}. For $m>1$, an arrow $A:=\langle v_1,x_1,v_2,\cdots ,v_{m-1}, x_{m-1},v_m\rangle$ is determined by the finite path $x_1x_2\cdots x_{m-1}$ of length $m-1$ on the graph, and we shall often regard it as such, by disregarding the contribution of the objects in the notation. In the category, the same arrow is equal to a composition
  \begin{equation}\label{eq: formal composition of morphisms on edges}
  A=A_{m-1}\circ A_{m-2}\circ\cdots \circ A_{2}\circ A_1
 \end{equation}
 of $m-1$ elementary arrows $A_s=\langle v_s,x_s,v_{s+1}\rangle$, $1\leq s\leq m$. For short, we write it as $x_{m-1}\circ \cdots \circ x_2\circ x_1.$ The category underlying the $k$-graph  will be obtained as a quotient category of $C$, as in \cite[Proposition 1, page 51]{MacL}.

  We let $\C_Y$ be the collection of arrows in $\C$ arising from paths on the digraph $DG$ belonging to $Y$,
 \begin{equation}\label{eq:def C_Y}
 \C_Y=\{\langle v,x,u,y,w\rangle\mid v,u,w\in V, x\in E_i, y\in E_j, 1\leq i\neq j\leq k \}.
 \end{equation}
 We refer to elements in $\C_Y$ as \emph{bicoloured} arrows.

    For objects $v,w$ in $\C$, we let $\C(v,w)$ be the set of morphisms in $\C$ which are arrows from $v$ to $w$. For each pair of objects $v,w\in \C$, we will define a binary relation $\mathcal{R}_{v,w}$ on the set of morphisms $\C(v,w)$. For this purpose, we introduce the following terminology: Given two arrows $A,A'$ in $\C$, we say that $A'$ is obtained from $A$ by \emph{diverting in $\C_Y$ over $\phi$} if  we have
  \begin{equation}\label{eq:generic arrow A}
  A:=\langle v_1,x_1,v_2,\cdots ,v_s, x_s, v_{s+1},x_{s+1},v_{s+2},\dots, v_{m-1}, x_{m-1},v_m\rangle
  \end{equation}
  and
  \begin{equation}\label{eq:generic arrow A' diverted from A}
  A':=\langle v_1,x_1,v_2,\cdots ,v_s, x_{s+1}', v_{s+1}',x_{s}',v_{s+2},\dots, v_{m-1}, x_{m-1},v_m\rangle
  \end{equation}
  with $\langle v_s, x_s, v_{s+1},x_{s+1},v_{s+2}\rangle, \langle v_s, x_{s+1}', v_{s+1}',x_{s}',v_{s+2}\rangle \in \C_Y$ and $x_{s+1}'x_{s}'=\phi(x_sx_{s+1})$, for some $1\leq s\leq m-1$ and $m>1$. Note that this makes sense on arrows of length at least $2$. Note also that diverting $A$ to $A'$ over $\phi$ does not change the number of edges of given colour in $A$. Further, since $\phi^2=1$, if $A'$ is obtained from $A$ by {diverting over $\phi$}, then also $A$ is obtained from $A'$ by {diverting in $\C_Y$ over $\phi$}: simply use that $x_{s}x_{s+1}=\phi(x_{s+1}'x_{s}')$.

  Let $v,w\in \C$. For $A,A'\in \C(v,w)$ we let
  \[
  A\mathcal{R}_{v,w}A'
   \]
   if $A'=A$ or there is a finite sequence $A_0=A,A_1, \dots ,A_n=A'$ of elements in $\C$ so that $A_{p+1}$ for $1\leq p\leq n-1$ is obtained from $A_p$ by diverting over $\phi$. The number of edges of same colour stays the same in each $A_p$.

  We claim that $R_{v,w}$ is an equivalence relation. The reflexivity $A\mathcal{R}_{v,w}A$ is clear for each $A$.

  To see that the relation is symmetric,
  let $A_0=A,A_1, \dots ,A_n=A'$  be a sequence of morphisms in $C$ implementing the relation $A\mathcal{R}_{v,w}A'$. Then
  the sequence $A_0'=A', A_1'=A_{n-1},\dots, A_{n-1}'=A_1, A_n'=A_0$ will implement $A'\mathcal{R}_{v,w}A$.

   For transitivity, suppose that $A\mathcal{R}_{v,w}A'$ and $A'\mathcal{R}_{v,w} A''$, which we may assume are represented by the sequences
   \[
   A_0=A,A_1, \dots ,A_n=A', B_0=A',B_1,\dots, B_p=A''.
   \]
   Then it is clear that $A_0=A,A_1,\dots,A_n=A', A_{n+1}=B_1,A_{n+2}=B_2,\dots, A_{n+p}=B_p$ implements $A\mathcal{R}_{v,w}A''.$

   We next claim that the equivalence relation $\mathcal{R}$ is preserved by composition of morphisms. Suppose that $B\in \C(v',v)$ and $A\mathcal{R}_{v,w}A'$. Let $A_0=A,A_1, \dots ,A_n=A'$ implement the equivalence between $A,A'$, then
   $C_0=A_0\circ B, C_1=A_1 \circ B,\dots,C_n=A_n\circ B$ implements the relation
   \[
 (A\circ B)\mathcal{R}_{v,w}(A'\circ B).
   \]
   If now $D\in \C(w,w')$, then similarly we have $(D\circ A)\mathcal{R}_{v,w'}(D\circ A')$. In all, $\mathcal{R}$ is a congruence in the sense of \cite[Page 52]{MacL}. Hence there is a category $\C/\mathcal{R}$ with object set $V$, the object set of $\C$, and set of morphisms $(\C/\mathcal{R})(v,w)$ equal to the quotient $\C(v,w)/\mathcal{R}_{v,w}$, for $v,w\in V$. We denote the class in $\C(v,w)/\mathcal{R}_{v,w}$ of a morphism $A\in \C(v,w)$ by $\dot{A}$.

  Next we show existence of the functor $d:\C/\mathcal{R}\to \N^k$. For an object $v$ in $\C/\mathcal{R}$ we let $d(v)=0$.  For an elementary arrow ${A}$ in $\C$ with  $A\in \C(v,w)$ for $v,w\in V$, there is a unique colour $i\in \{1,\dots,k\}$ of the edge underlying the arrow. We set $d({A}):=e_i$.

   We extend this to an arbitrary arrow ${A}$ in $\C$
by
\[
d({A}):=e_{i_1}+e_{i_2}+\cdots+e_{i_n}\text{ if }A=\langle v_1,x_1,v_2,\cdots, x_{n},v_{n+1}\rangle, x_s\in E_{i_s}, 1\leq s\leq n,
\]
where $i_s$ are not necessarily  distinct. If two arrows $A,A'$ are given as in \eqref{eq:generic arrow A} and \eqref{eq:generic arrow A' diverted from A} with $A\mathcal{R}_{v,w}A'$, we have that $d(A)=d(A')$ because at the vertex $v_s$ where the path in $A$ is diverted we have $e_{i_s}+e_{i_{s+1}}=e_{i_{s+1}}+e_{i_s}$. Let
\begin{equation}\label{eq:functor H to Nk}
  n_K=\begin{cases}\#\{s\mid 1\leq s\leq n,i_s=K\}&\text{ if }K\in \{i_1,i_2,\dots,i_{n}\}\\
  0&\text{ if }K\not\in \{i_1,i_2,\dots,i_{n}\}
  \end{cases}
\end{equation}
for $K=1,\dots, k$. Then
\[
d(\dot{A}):=(n_1,n_2,\dots,n_k)\in \N^k
\]
is well-defined. Moreover, it satisfies $d(\dot{A}\dot{B})=d(\dot{A})+d(\dot{B})$ for $\dot{A}\in \C(v,w)/\mathcal{R}_{v,w}$ and $\dot{B}\in \C(v',v)/\mathcal{R}_{v',v}$. Likewise for composition on the right. This defines the functor $d:\C/\mathcal{R}\to \N^k$. We will show that $d$ enjoys the unique factorisation property required of a higher rank graph.

For this we fix an ordering of the colours and show, as an intermediate step, that every morphism $\dot{A}$ has a representative $A$ in $\C$ with all edges of the same colour grouped together. For simplicity of notation, we assume that the colours appear in the order $\{1,2,\dots,k\}$. We claim that each class $\dot{A}$ contains a representative
\begin{equation}\label{eq:rainbow form of colours}
A=\langle v_1,x_1,\cdots, x_{n_1},v_{n_1+1},x_{n_1+1},\cdots,x_{n_1+n_2},\cdots, v_m\rangle,
\end{equation}
where  $x_1,\dots,x_{n_1}$ are edges of colour $i_1$, followed by edges $x_{n_1+1},\dots, x_{n_1+n_2}$ of colour $i_2$, and so on, with $n_K$ edges of colour $i_K$ at the end, where $i_1>i_2>\cdots >i_K$ are distinct colours in $\{1,\dots,k\}$. Thus, the largest colour appears nearest to the origin of the path on the digraph that determines the arrow, and the colours appear in decreasing order along the path towards it terminus.

Note that \eqref{eq:rainbow form of colours} is trivially satisfied if $\dot{A}$ is the class of an arrow $A$ on the digraph that only follows one colour. Assume next that there are only two colours $i>j$ in a representative $A$ for $\dot{A}$.   Let $v_s$ be the first vertex at which we have a bicoloured path $\langle v_s, x_s,v_{s+1},x_{s+1},v_{s+2}\rangle$  with $x_s\in E_j,x_{s+1}\in E_i$. Then we divert this into the equivalent path $\langle v_s, x_{s+1}',v_{s+1}',x_{s}',v_{s+2}\rangle$ for unique $x_s'\in E_j$ and $x_{s+1}'\in E_i$. If there are no more edges from  $E_i$ past the vertex $v_{s+2}$, we are done. If not, continue the process until we are left with a representative for $\dot{A}$ of the form required in \eqref{eq:rainbow form of colours}.

Assume now that there are three distinct colours $l>j>i$ in a representative $A$ for $\dot{A}$. If all three colours appear in decreasing order $l,j,i$ in the path supporting $A$, there is nothing to prove. If two colours appear in reverse order at a time, we reduce to the previous case. Assume now that a tricoloured path appears in increasing order of colours. For simplicity of notation, we may assume that this is an arrow  in $\C$ of the form
\[
\langle v_1,x,v_2,y,v_3,z,v_4\rangle, v_1,\dots, v_4\in V, x\in E_i, y\in E_j, z\in E_l.
\]
This is the composition $\langle v_3,z,v_4\rangle \circ\langle v_2,y,v_3\rangle \circ \langle v_1,x,v_2\rangle$ (or $z\circ y\circ x$) of elementary arrows in colours $l>j>i$, and we claim that it is the same in $\C/\mathcal{R}$ as a (unique) arrow of the form
\begin{equation}\label{eq:three colours in increasing order in degree}
\langle v_1, \bar{z}, v_2',\bar{y}, v_3',\bar{x}, v_4\rangle, \bar{z}\in E_l, \bar{y}\in E_j, \bar{x}\in E_i,
\end{equation}
where $v_2',v_3'\in V$.  By successive application of this reversing of order, it will follow that $\dot{A}$ admits a representative with the colours appearing as in \eqref{eq:rainbow form of colours}.

Since $\langle v_1,x,v_2,y,v_3\rangle \in \C(v,w)\cap \C_Y$, there are unique $x^1\in E_i,y^1\in E_j$ such that
\[
\langle v_1,x,v_2,y,v_3\rangle \mathcal{R}_{v_1,v_3} \langle v_1, y^1, u_1, x^1,v_3\rangle,
\]
with $u_1=t(y^1)=o(x^1)$. For simplicity, write this as $(y\circ x) \mathcal{R}_{v_1,v_3} (x^1\circ y^1)$.  Continuing this way,  the bijection $\phi$ prescribes  edges $x_1, x_2,x^2\in E_i$, $y_1,y_2,y^2\in E_j$ and $z_1,z^1,z_2,z^2\in E_l$ such that
\begin{align*}
(z\circ x^1) &\mathcal{R}_{u_1,v_4}\, (x^2\circ z^1),\quad \text{ with }o(x^2)=t(z^1)=u_2\\
(z^1\circ y^1) \,&\mathcal{R}_{v_1,u_2}\,(y^2\circ z^2), \quad \text{ with }o(y^2)=t(z^2)=u_3\\
(z\circ y) \,&\mathcal{R}_{v_2,v_4}\,(y_1\circ z_1),\quad \text{ with }o(y_1)=t(z_1)=u_4\\
(z_1 \circ x)\,&\mathcal{R}_{v_1,u_4}\,(x_1\circ z_2),\quad \text{ with }o(x_1)=t(z_3)= u_5\\
(y_1\circ x_1)\, &\mathcal{R}_{u_3,v_4}\,(x_2\circ y_2),\quad \text{ with }o(x_2)=t(y_2)=u_6,
\end{align*}
where $u_1=o(z^1)$ and $u_4=t(x_1)$.
By our assumption (F2) on $\operatorname{DG}$, we have $u_2=u_6$ and $u_3=u_5$ and
\begin{align*}
\bar{x}:=x_2=x^2,\\
\bar{y}:=y_2=y^2,\\
\bar{z}:=z_2=z^2.
\end{align*}
With $v_2':=u_3$ and $v_3':=u_2$, this gives the claimed representative in \eqref{eq:three colours in increasing order in degree}, where we have used that the relation $\mathcal{R}$ preserves composition of morphisms. Successive applications of this reversing of order in a tricoloured path show that $\dot{A}$ admits a representative as in \eqref{eq:rainbow form of colours}. Note at the same time that the class in $\C/\mathcal{R}$ of the morphism $ \bar{x}\circ\bar{y}\circ \bar{z}$   decomposes uniquely as products of two morphisms along any choice in $\mathcal{R}_{v_1,v_4}$ which involves tricoloured paths. More precisely, the decompositions are
\begin{align}
z\circ(y\circ x)&\text{ corresponding to } e_l+(e_j+e_i)\label{eq:lji};\\
z\circ(x^1\circ y^1)&\text{ corresponding to } e_l+(e_i+e_j)\label{eq:lij};\\
\bar{x}\circ(z^1\circ y^1)&\text{ corresponding to } e_i+(e_l+e_j)\label{eq:ilj};\\
\bar{x}\circ(\bar{y}\circ\bar{z})&\text{ corresponding to } e_i+(e_j+e_l)\label{eq:ijl};\\
y_1\circ(x_1\circ\bar{z})&\text{ corresponding to } e_j+(e_i+e_l)\label{eq:jil};\\
y_1\circ(z_1\circ x)&\text{ corresponding to } e_j+(e_l+e_i)\label{eq:jli}.
\end{align}

If more than three colours appear in a representative $A$ for $\dot{A}$, say $i_1<i_2<\cdots <i_K$ with $K\geq 4$, and if a path supporting $A$ has colours in increasing order then we resort to the previous cases.  Thus, if $K=4$ and a path appears with colours in the order $i_1<i_2<i_3<i_4$,  we first reverse the path onto colours $i_1, i_4, i_3,i_2$, working from the terminus of the path (source of its arrow in the category) to the left towards its origin. Then we move the edges of colour $i_1$ past the ones of colours $i_4,i_3,i_2$, using the previous cases. Similarly for $K>4$. In all, \eqref{eq:rainbow form of colours} follows.

Now we are ready to prove the factorisation property of $d$. Suppose that $\dot{A}$ is in $\C/\mathcal{R}$ with $d(\dot{A})=(n_1,n_2,\dots,n_k)\in N^k$. We must show that whenever $(n_1,n_2,\dots,n_k)=(m_1,m_2,\dots,m_k)+(p_1,p_2,\dots,p_k)$ in $\N^k$ there are unique morphisms $\dot{B},\dot{C}$ in $\C/\mathcal{R}$ so that
\[
\dot{A}=\dot{B}\dot{C}, d(\dot{B})=(m_1,m_2,\dots,m_k)\text{ and }d(\dot{C})=(p_1,p_2,\dots,p_k).
\]

The proof is structured into cases determined by the number of non-zero entries $n_s$, that is, by the number of colours that appear in a morphism in the class $\dot{A}$.

\noindent{{Case 1: single colour.}} Thus $n_i>0$ for a unique $1\leq i\leq k$. If $n_i=1$, then by our construction of $\C/\mathcal{R}$ and $d$ we know that $\dot{A}$ is the class of an elementary arrow
$A=\langle v_1,x_1,v_2\rangle$ with $x_1\in E_i$ and only the trivial decomposition involving identities at $v_1,v_2$ is possible. If $n_i>1$, then a representative for $A$ consists of a path of length $n_i$ along edges in $E_i$, and so we can decompose $\dot{A}=\dot{B}\dot{C}$ with $d(\dot{B})=m_i, d(\dot{C})=p_i$ for any choice $n_i=m_i+p_i$ in $\N$.

\noindent{{Case 2: two colours.}} By our earlier claim \eqref{eq:rainbow form of colours}, we may assume that $d(\dot{A})=(n_1,n_2,\dots,n_k)$ with $n_i\geq1, n_j\geq 1$ at  $i>j$ in $\{1,\dots, k\}$ and $n_l=0$ at all other entries.

\noindent{{Case 2.1: $n_i=n_j=1$}}.
By definition of $\C/\mathcal{R}$, the morphism $\dot{A}$ is the class of a bicoloured arrow $\langle v_1,x_1,v_2,x_2,v_3\rangle$, with $x_1\in E_i$ and $x_2\in E_j$. Let $x_2'\in E_j$ and $x_1'\in E_i$ so that $\phi(x_1x_2)=x_2'x_1'$, and put $v_2'=t(x_2')=o(x_1')$. Then
\[
A=\langle v_2', x_1',v_3\rangle\langle v_1,x_2',v_2'\rangle=\langle v_2,x_2,v_3\rangle \langle v_1,x_1,v_2\rangle
\]
give the unique decompositions $\dot{A}=\dot{B}\dot{C}$ according to the decompositions  $e_i+e_j$ and $e_j+e_i$  of $d(\dot{A})$.

\noindent{{Case 2.2: $n_i>1, n_j\geq 1$}}. Without loss of generality, we may assume that $n_i\geq n_j$.

The given morphism $\dot{A}$ is the class of an arrow following a path with $n_i$ edges of colour $i$ and $n_j$ edges of colour $n_j$.    Let $n_i=m_i+p_i$ and $n_j=m_j+p_j$, and consider first the case that $m_i,p_j\geq 1$. We then divert the path in  $A$ after $p_i+m_i-1$ edges of colour $i$ by applying case 2.1 to obtain a representative for $\dot{A}$ in the form of a path with $p_i+m_i-1$ edges of colour $i$, then one edge of colour $j$ followed by an edge of colour $i$, and finally with $p_j+m_j-1$ edges of colour $j$ at the end. If $p_i+m_i-2=0$ and $p_j+m_j-2=0$, we are done, having recovered case 2.1. Otherwise, if for example $p_i+m_i-2\geq 1$, then apply case 2.1 to divert an edge of colour $j$ onto one of colour $i$, thus obtaining a representative with $p_i+m_i-2$ edges of colour $i$, then one edge of colour $j$ followed by two of colour $i$, and finally $m_j+p_j-1$ edges of colour $j$. If also $m_j+p_j-2\geq 1$, we divert another edge of colour $j$ successively past the two of colour $i$ to get $p_i+m_i-2$ of colour $i$, two of colour $j$, two of colour $i$, and finally $p_j+m_j-2$ of colour $j$.  Continuing this process, we divert all $m_i$ edges of colour $i$ followed by the $p_j$ edges of colour $j$ into a path where there are $p_j$ edges of colour $j$ first, followed by $m_i$ edges of colour $i$. This determines the required decomposition of $\dot{A}=\dot{B}\dot{C}$, where
$d(\dot{B})=m_ie_i+m_je_j$ and $d(\dot{C})=p_ie_i+p_je_j$.

The remaining cases where only two colours are present are treated similarly.

\noindent{{Case 3: three colours.}}  By our earlier claim, we may assume that $d(\dot{A})=(n_1,n_2,\dots,n_k)$ with $n_i\geq 1, n_j\geq  1, n_l\geq 1$ at $i>j>l$ in $\{1,\dots, k\}$ and $n_h=0$ at all other entries $h\in \{1,\dots,k\}$. Assuming $n_i=m_i+p_i$, $n_j=m_j+p_j$ and $n_l=m_l+p_l$, then by considering cases as before and applying the factorisations \eqref{eq:lji}-\eqref{eq:jli} accordingly results in a factorisation  $\dot{A}=\dot{B}\dot{C}$ with $d(\dot{B})=m_ie_i+m_je_j+m_le_l$ and $d(\dot{C})=p_ie_i+p_je_j+p_le_l$. We omit the details.

In the general case where $d(\dot{A})$ has nonzero entries in more than three distinct colours $i_1>i_2>\cdots >i_K$, $K\geq 4$, the problem reduces to decomposing along any expression of $(n_1,n_2,\dots,n_K)$ where at least one of the entries $n_{i_s}$, $s=1,\dots, K$, is expressed as a sum  $m_{i_s}+p_{i_s}$ for $m_{i_s},p_{i_s}\geq 1$. This reduces to cases 2 and 3.
 \end{proof}

We  remark that an alternative argument to construct a higher rank graph from the input of a $k$-dimensional digraph could be obtained by appealing to \cite{HRSW}. The main step is to identify  a complete collection of squares that is associative based on $\operatorname{DG}$ and then an application of \cite[Theorem 4.4]{HRSW} provides the desired $k$-graph.

\begin{remark}  It is possible to express the bijection $\phi$ using notation from \cite{KP}. If $C$ and $D$ are directed $1$-graphs with common set of vertices $V=C^0=D^0$, distinct sets of edges $C^1$, $D^1$, and commuting vertex matrices,  let
\[
C^1\ast D^1=\{(x,y)\in C^1\times D^1\mid t(x)=o(y)\}.
\]
Then the bijection $\phi$ in (F1) is given by its restrictions $\phi_{i,j}:E_i^1\ast E_j^1\to E_j^1\ast E_i^1$, for all  $i\neq j$ in $\{1,2,\dots,k\}$. This construction is reminiscent of an older idea of a product of two (possibly directed) graphs, as described in Ore's monograph \cite{Ore}.
\end{remark}

We note that condition (F2) in Definition~\ref{def:k digraph from coloured graph} is vacuous when $k=2$. We shall  refer to $E=E_1\sqcup E_2\sqcup \dots \sqcup E_k$ as the \emph{$1$-skeleton} of $\Lambda_{\operatorname{DG}}$.

\begin{theorem}\label{thm:one vertex k-digraph from k-cube}
For any  complex $\mathcal{X}$ with universal cover equal to the product of $k$ regular trees, where $k\geq 2$,  there is a  $k$-dimensional digraph $\operatorname{DG}(\mathcal{X})$ defined by sending a vertex of $\mathcal{X}$ to a vertex in $\operatorname{DG}(\mathcal{X})$,  and by sending  each geometric edge in $\mathcal{X}$ to two edges in the edge set of $\operatorname{DG}(\mathcal{X})$, of same colour and opposite orientations.
\end{theorem}

We prove this theorem in stages. First we prove the one-vertex case for $k=2$, where the statement is a consequence of the description of a one-vertex $2$-complex as $S_{A,B}$ for a VH-structure $A,B$. Then we prove the case $k=3$ by employing geometric cubes. First we record a key consequence whose proof is immediate from Theorem~\ref{thm:the category associated to a digraph}
 and Theorem~\ref{thm:one vertex k-digraph from k-cube}.

\begin{corollary}\label{cor:the one vertex k-graph associated to a digraph}
Given $\mathcal{X}$ a $k$-cube complex with $k\geq 2$,  there is a $k$-graph $\Lambda_{\operatorname{DG}(\mathcal{X})}$ with vertex set equal to the vertex set of $\mathcal{X}$ and whose $1$-skeleton contains two edges for each geometric edge of $\mathcal{X}$, in a colour preserving and orientation reversing assignment. Further, $\Lambda_{\operatorname{DG}(\mathcal{X})}$ is row-finite.
\end{corollary}

\begin{definition}\label{def:the k-graph C star algebra of P}
Given a $k$-cube complex $\mathcal{X}$, its associated $C^*$-algebra is the higher rank graph $C^*$-algebra $C^*(\Lambda_{\operatorname{DG}(\mathcal{X})})$.
\end{definition}
To simplify the notation we write  $\Lambda(\mathcal{X})$ in place of  $\Lambda_{\operatorname{DG}(\mathcal{X})}$. Continuing our convention from the proof of Theorem~\ref{thm:the category associated to a digraph},  we use letters such as $x,y$ to denote both generic elements in the edge set $E$ of $\operatorname{DG}(\mathcal{X})$ and their corresponding morphisms in $\Lambda(\mathcal{X})$ associated to elementary arrows. In particular, a bicoloured path $xy$ on the digraph with $x\in E_i, y\in E_j$ for distinct colours $i\neq j$ will be $y\circ x$ as  composition as morphisms in the $k$-graph.

As another point of notation, in the proof of the next result and that of Theorem~\ref{thm:one vertex k-digraph from k-cube} in the one-vertex case, we shall distinguish between labels such as $a$ in  a generating subset $A$ of a $k$-cube group $G$ and the label of the edge it defines in the associated $k$-digraph, for which we reserve notation of the form $\mathfrak{a}$. In the group we have $aa^{-1}=1=1_G$, so in the one-vertex complex $a^{-1}$ means orientation reversal of the edge labelled with $a$, while in the $k$-dimensional digraph the edge labelled $a$ is sent into distinct directed edges $\mathfrak{a}_1$, $\mathfrak{a}_2$ (with no cancellation inherited from the group).

\begin{lemma}\label{prop:2-graph from VH structure} Assume that $S_{A, B}$ is a one-vertex square complex with associated group $G$ given by a VH-structure $(A,B)$ with $\#A$ and $\#B$ both even positive integers. Suppose that
\begin{equation}\label{eq: A B expanded notation}
A=\{a_1,\dots,a_L,a_{L+1},\dots,a_{2L}\}\text{ and }B=\{b_1,\dots,b_K, b_{K+1},\dots ,b_{2K}\},
\end{equation}
with $a_ra_{L+r}=1$ in $G$ for all $r=1,\dots,L$ and $b_sb_{K+s}=1$ in $G$ for all $s=1,\dots,K$, with $K,L\geq 1$. In particular, for each $r=1,\dots,L$, we have that $a_r$ and $a_{L+r}$ label the same geometric edge in $S_{A,B}$, but with opposite orientations. Similarly for $b_s$, $b_{K+s}$ with $s=1,\dots,K$.

 Then there is a  $2$-dimensional digraph in the sense of Definition~\ref{def:k digraph from coloured graph} with edge set $E(S_{A,B})=E_1\sqcup E_2$ obtained by associating to each $a_r$ for $r=1,\dots, 2L$ a directed edge $\mathfrak{a}_r$ in $E_1$, and to each $b_s$ for  $s=1,\dots,2K$  a directed edge $\mathfrak{b}_s$ in $E_2$.
\end{lemma}

\begin{proof} We have that for each $a_r\in A$ and  $b_s\in B$, with $r=1,\dots, 2L$ , $s=1,\dots, 2K$, there are unique $a_{l(r,s)}\in A$ and $b_{m(r,s)}\in B$, with $l(r,s)\in\{1,\dots, 2L\}$ and $m(r,s)\in\{1,\dots, 2K\}$, such that
\[
a_rb_s=b_{m(r,s)}a_{l(r,s)}.
\]
In particular, $a_rb_s$ is contained as a vertical-horizontal pair of edges in a unique geometric square in the family
\[
S_*^{a_r,b_s}, r=1,\dots, 2L, s=1,\dots ,2K, *=O,H, V,R,
\]
with $b_{m(r,s)}a_{l(r,s)}$ forming the  horizontal-vertical pair of edges in $S_*^{a_r,b_s}$ starting and ending at the same vertices as $a_rb_s$ (which we recall coalesce to the single vertex $v$ of $S_{A,B}$).

Let $xy$ be a path of length two with $x\in E_1$ and $y\in E_2$. Then $xy$  is uniquely determined by $x=\mathfrak{a}_r$ for some $r=1,\dots, 2L$ and $y=\mathfrak{b}_s$ for some $s=1,\dots,2K$. Let now $a_r,b_s, a_{l(r,s)}$ and $b_{m(r,s)}$ be as above. Then $y':=\mathfrak{b}_{m(r,s)}\in E_2$ and $x':=\mathfrak{a}_{l(r,s)}\in E_1$ determine a unique path of length two $y'x'$ so that $xy\sim y'x'$.

This defines the required bijection $\phi:Y\to Y$  with $\phi(xy)=y'x'$ on the set $Y$ of all paths of length two of distinct colours.
\end{proof}

\begin{example}\label{ex:complex from torus}
A simple construction of a $2$-graph based on the procedure of Corollary~\ref{cor:the one vertex k-graph associated to a digraph} recovers a known example, see \cite{Pow}, \cite{Y} and \cite[Example 11.1(1)]{LV}, where $\theta(i,j)= (i,j)$ is the identity permutation of the set $\{1,2\}\times \{1,2\}$. Consider the $(2,2)$-group  $G=\Z\times \Z$ with generating sets $A=\{a,a^{-1}\}$ corresponding to the first copy of $\Z$ and $B=\{b, b^{-1}\}$ for the second copy. We have the commutation relation $ab=ba$ as the basis for a geometric square $S_O^{a,b}$. The one-vertex complex has two loops. An application of Lemma~\ref{prop:2-graph from VH structure} yields a  $2$-dimensional digraph with edge set a disjoint union of $E_1=\{\mathfrak{a}_1,\mathfrak{a}_2\}$ and $E_2=\{\mathfrak{b}_1,\mathfrak{b}_2\}$, thus four loops, with the  bijection $\phi$ on the set of paths of length two of distinct colours read off from $S_O^{a,b}, S_H^{a,b}, S_V^{a,b}$ and $S_R^{a,b}$ as follows:
\[
\mathfrak{a}_1\mathfrak{b}_1\sim\mathfrak{b}_1\mathfrak{a}_1, \mathfrak{a}_1\mathfrak{b}_2\sim\mathfrak{b}_2\mathfrak{a}_1, \mathfrak{a}_2\mathfrak{b}_1\sim\mathfrak{b}_1\mathfrak{a}_2\text{ and }\mathfrak{a}_2\mathfrak{b}_2\sim\mathfrak{b}_2 \mathfrak{a}_2.
\]
\end{example}

\begin{example}\label{ex:group 37} We now present a $2$-graph from this recipe where the group $G$ is not of product type. As we will  explain, figure~\ref{figure four squares} shows an example of a $(4,4)$-group $G$, cf. \cite{vdovina-YB}.

\begin{figure}[h]
\centering
\includegraphics[height=2.7cm]{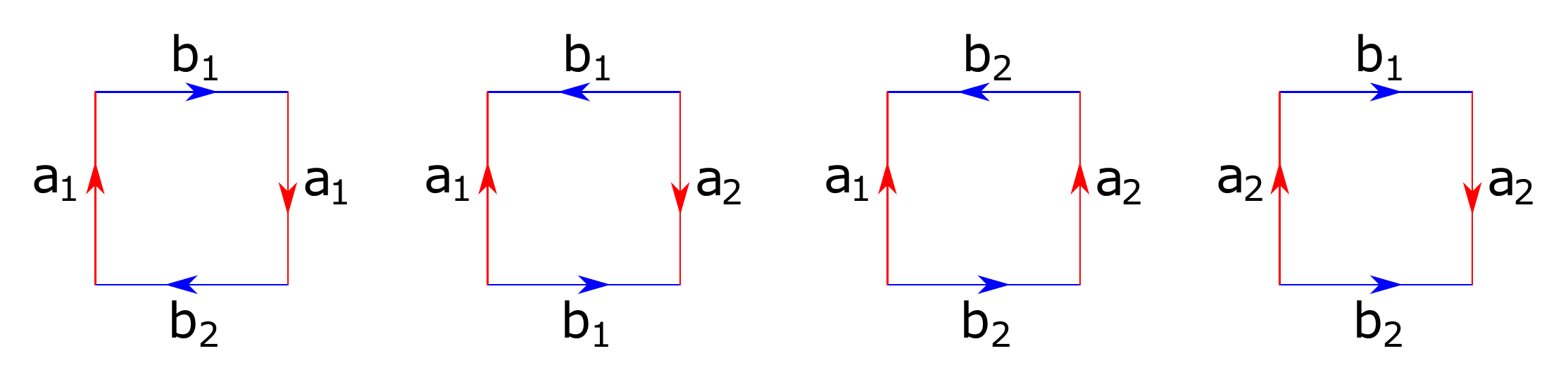}
\caption{A concrete example of $(4,4)$-group from four squares}\label{figure four squares}
\end{figure}

The four squares are geometric squares representing the $2$-cells of an associated complex $S_{A, B}$, where $A=\{a_1,a_2,a_3,a_4\}$ for $a_3=a_1^{-1}$ and $a_4=a_2^{-1}$, and $B=\{b_1,b_2,b_3,b_4\}$ for $b_3=b_1^{-1}$ and $b_4=b_2^{-1}$. Here $L=K=2$, cf. Lemma~\ref{prop:2-graph from VH structure}. With our convention in \eqref{eq:S ab}  we have, from left to right, $S_O^{a_1, b_1}$, $S_O^{a_1,b_{3}}$, $S_O^{a_1, b_{4}}$ and $S_O^{a_2,b_1}$.

The associated $2$-graph $\Lambda(S_{A,B})$ from Corollary~\ref{cor:the one vertex k-graph associated to a digraph}
 has $1$-skeleton determined by the $2$-dimensional digraph whose edges are given by the disjoint union of
$E_1=\{\mathfrak{a}_1, \mathfrak{a}_2, \mathfrak{a}_3,\mathfrak{a}_4\}$ and $ E_2=\{\mathfrak{b}_1,\mathfrak{b}_2, \mathfrak{b}_3,\mathfrak{b}_4 \}$, see Lemma~\ref{prop:2-graph from VH structure}.  Let us now describe explicitly the bijection $\phi:Y\to Y$. Note that the digraph has 4 loops of one colour (red) and 4 loops of the second colour (blue).

We have 16 paths of length two of the form $xy$, where $x\in E_1$ and $y\in E_2$, given by all the possible $\mathfrak{a}_i\mathfrak{b}_j$ with $i,j=1,\dots, 4$. Correspondingly, we have all possible vertical-horizontal pairs of edges $a_ib_j$  in the collection  of geometric squares
\[
S_*^{a_j,b_j},i,j=1,\dots,4, *=O,V,R,H.
\]
Pick for each $a_ib_j$ the unique $a_{l(i,j)}\in A$ and $b_{m(i,j)}\in B$ such that $b_{m(i,j)}a_{l(i,j)}$ is the corresponding horizontal-vertical pair of edges in the same square, and let $y'=\mathfrak{b}_{m(i,j)}$, $x'=\mathfrak{a}_{l(i,j)}$ as prescribed by the proof of Lemma~\ref{prop:2-graph from VH structure}.

Explicitly, corresponding to the horizontal-vertical pairs of edges in the geometric square
\[
\{S_O^{a_1,b_1}, S_V^{a_1,b_1},  S_R^{a_1,b_1},  S_H^{a_1,b_1}\},
\]
it is seen that $\phi(\mathfrak{a}_1\mathfrak{b}_1)=\mathfrak{b}_4\mathfrak{a}_3$, $\phi(\mathfrak{a}_3\mathfrak{b}_4)=\mathfrak{b}_1\mathfrak{a}_1$, $\phi(\mathfrak{a}_1\mathfrak{b}_2)=\mathfrak{b}_3\mathfrak{a}_3$ and $\phi(\mathfrak{a}_3\mathfrak{b}_3)=\mathfrak{b}_2\mathfrak{a}_1$.

Similarly, from the geometric square
\[
\{S_O^{a_1,b_3}, S_V^{a_1,b_3}, S_R^{a_1,b_3}, S_H^{a_1,b_3}\}
\]
we get $\phi(\mathfrak{a}_1\mathfrak{b}_3)=\mathfrak{b}_1\mathfrak{a}_4$, $\phi(\mathfrak{a}_3\mathfrak{b}_1)=\mathfrak{b}_3\mathfrak{a}_2$, $\phi(\mathfrak{a}_2\mathfrak{b}_3)=\mathfrak{b}_1\mathfrak{a}_3$ and $\phi(\mathfrak{a}_4\mathfrak{b}_1)=\mathfrak{b}_3\mathfrak{a}_1$;
from
\[
\{S_O^{a_1,b_4}, S_V^{a_1,b_4},  S_R^{a_1,b_4},  S_H^{a_1, b_4}\}
\]
we get $\phi(\mathfrak{a}_1\mathfrak{b}_4)=\mathfrak{b}_2\mathfrak{a}_2$, $\phi(\mathfrak{a}_3\mathfrak{b}_2)=\mathfrak{b}_4\mathfrak{a}_4$, $\phi(\mathfrak{a}_4\mathfrak{b}_4)=\mathfrak{b}_2\mathfrak{a}_3$ and $\phi(\mathfrak{a}_2\mathfrak{b}_2)=\mathfrak{b}_4\mathfrak{a}_1$;
finally,  from the geometric square
\[
\{S_O^{a_2,b_1}, S_V^{a_2,b_1},  S_R^{a_2,b_1}, S_H^{a_2,b_1}\}
\]
we get $\phi(\mathfrak{a}_2\mathfrak{b}_1)=\mathfrak{b}_2\mathfrak{a}_4$, $\phi(\mathfrak{a}_4\mathfrak{b}_2)=\mathfrak{b}_1\mathfrak{a}_2$, $\phi(\mathfrak{a}_2\mathfrak{b}_4)=\mathfrak{b}_3\mathfrak{a}_4$ and $\phi(\mathfrak{a}_4\mathfrak{b}_3)=\mathfrak{b}_4\mathfrak{a}_2$;
 this describes the bijection $\phi$ completely.

 The link of $S_{A,B}$ at its vertex $v$ is the complete bipartite graph of type $(4,4)$, see Figure~\ref{figure link}.

 \begin{figure}[h]
\centering
\includegraphics[height=4cm]{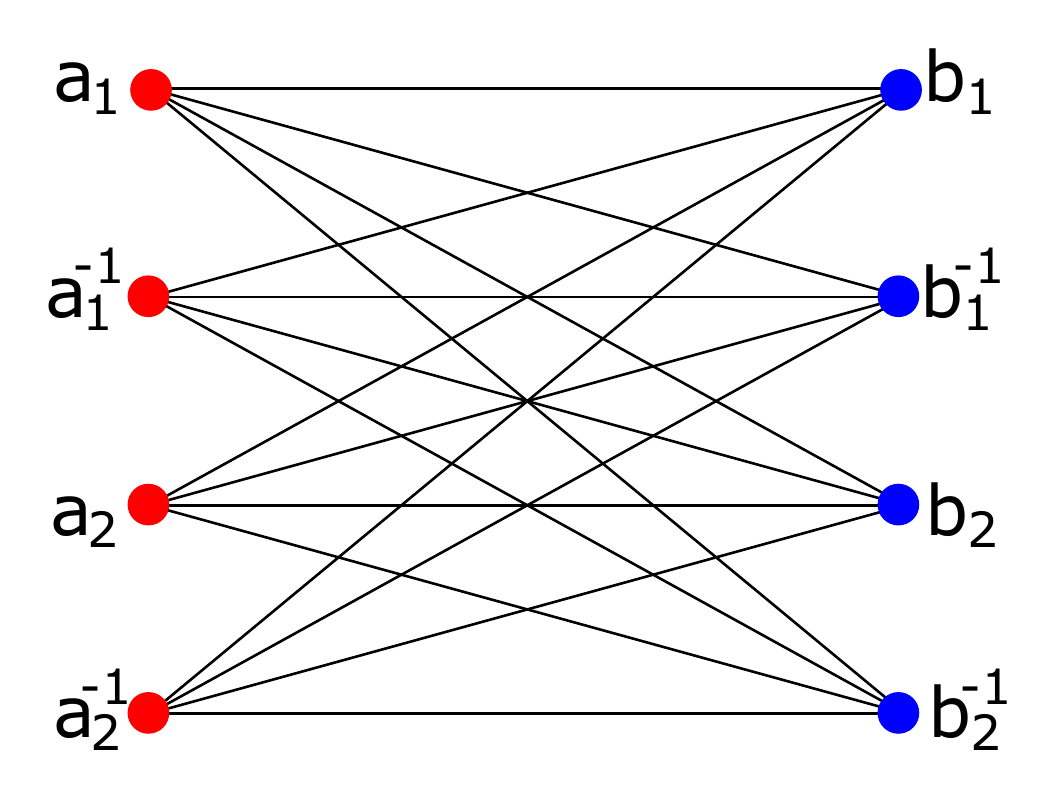}
\caption{The link of the complex}\label{figure link}
\end{figure}

It follows that $G:=\pi_1(S_{A,B},v)$ is a $(4,4)$-group. In fact, $G$ is the fundamental group of a CAT(0) complex with Gromov link condition, see \cite{Gromov}. We remind that every edge of the complex belongs to four squares, see figure~\ref{figure universal cover} for a fragment of the universal cover of the complex showing the edge $a_1$ belonging to four squares (in the universal cover).
\begin{figure}[h]
 \includegraphics[height=5cm]{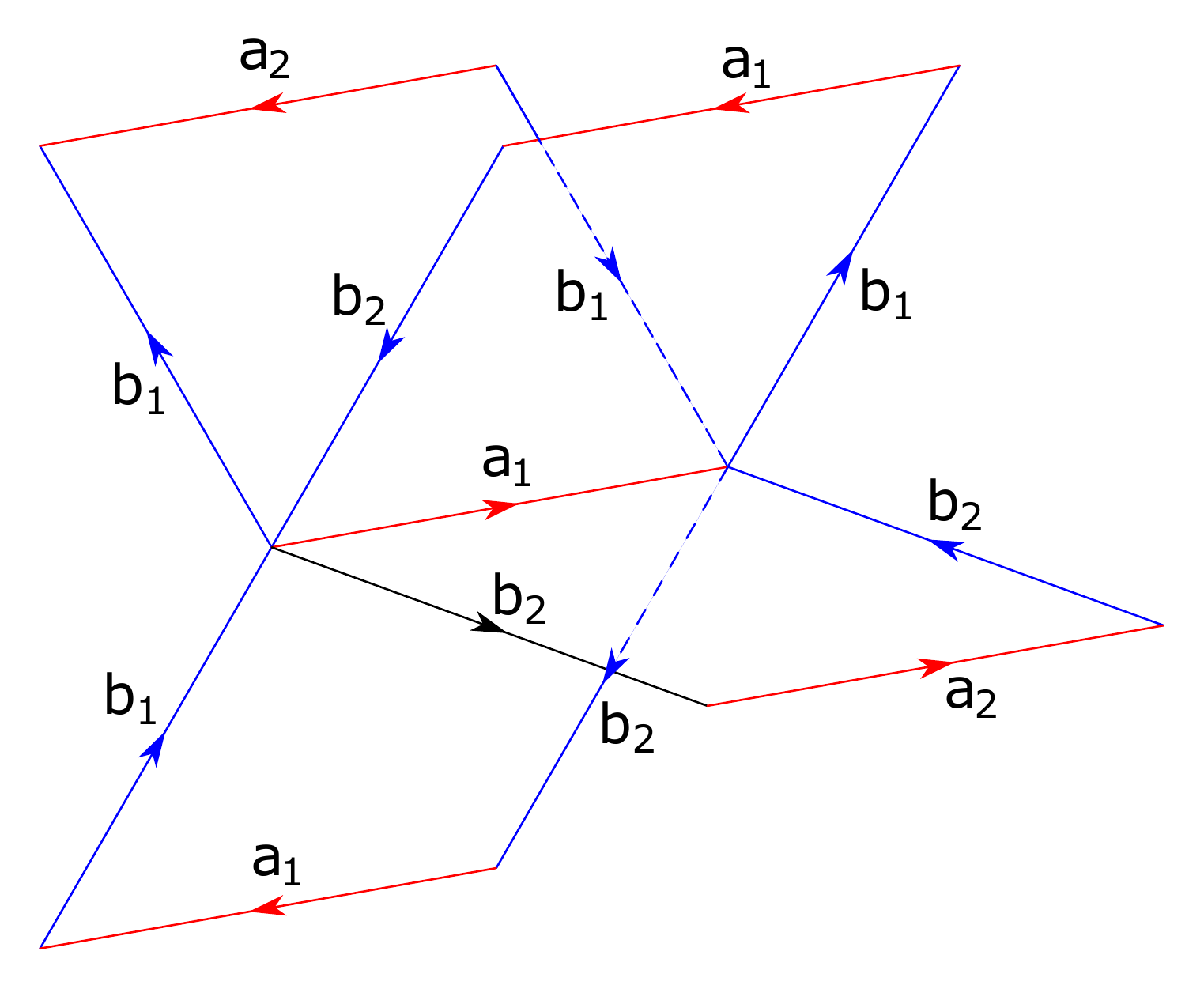}
\caption{Fragment of the universal cover showing the edge $a_1$}
\label{figure universal cover}
 \end{figure}
\end{example}

{\begin{remark}\label{rem:our 2 graph example2.37 versus KR}
The group with the same VH-structure as in Figure~\ref{figure four squares} appears also in \cite{KR}, Section 7, as the group $2\times2.37$ in their list. However, their $2$-graph is different from the one in Example~\ref{ex:group 37}, since, if we translate
the notions of \cite{KR} into higher rank graphs, the $2$-graph corresponding to the group $2\times2.37$ would have sixteen vertices.
In general, the $2$-graphs of \cite{KR} corresponding to $(2m,2n)$ groups give $2$-graphs with $4mn$ vertices, $4(m-1)mn$ edges of one colour and $4mn(n-1)$ edges of another colour. The $2$-graphs of \cite{R}, \cite{RS2} have $3(q^2+q+1)$ vertices and $3(q^2+q+1)q$ edges of each colour for $q$ being a prime power different from $3$.
\end{remark}}

\begin{proof} (Proof of Theorem~\ref{thm:one vertex k-digraph from k-cube}, the general case.) Fix $k\geq 3$. Assume first that we have a one-vertex $k$-cube complex $P$, thus we may take it of form $P_{A_1,\dots,A_k}$  associated with a $k$-cube group $G$ with underlying structure determined by the ordered tuple $(A_1,\dots,A_k)$. As explained, the edges of the complex are labelled by the generators of $G$. For each $i=1,\dots,k$ we write
\[
A_i=\{a_1^i,\dots,a_{L_i}^i,a_{L_{i}+1}^i,\dots,a_{2L_{i}}^i\},
\]
with the convention that $a_r^i a_{L_i+r}^i=1$ in $G$ for $1\leq r\leq L_i$. Define a digraph with edge set  $E=E_1\sqcup E_2\sqcup\dots \sqcup E_k$ by the assignment that to each $a_r^i$ corresponds a directed edge $\mathfrak{a}_r^i$ in $E_i$, with $i=1,\dots ,k$ and $r=1,\dots,2L_i$. We must identify a bijection $\phi$ on the set of bicoloured edges and establish conditions (F1) and (F2) of Definition~\ref{def:k digraph from coloured graph}.

Suppose that $xy$ is a path of length two in $E$ with $x\in E_i$ and $y\in E_j$ for distinct $i,j$ in $\{1,\dots,k\}$. Then there are  $a_r^{i}\in A_i$ and $b_s^j\in A_j$ for unique $r=1,\dots, 2L_i$ and $s=1,\dots,  2L_j$, such that $x=\mathfrak{a}_r^i$ and $y=\mathfrak{a}_s^j$. Since we have a cube complex, there is an associated square complex $P_{A_i, A_j}$ with corresponding group $G_{A_i,A_j}$, where we choose the convention that $A_i$ is vertical and $A_j$  horizontal direction. Lemma~\ref{prop:2-graph from VH structure} implies that there is a unique path of length two $y'x'$ with $y'\in A_j$ and  $x'\in A_i$, corresponding to a unique square with vertical-horizontal and horizontal-vertical pairs given by
\[
a_r^{i}a_s^{j}=a_{m(r,s)}^j a_{l(r,s)}^i
\]
such that $xy\sim y'x'$. Here $a_{m(r,s)}^j \in A_j$ and $a_{l(r,s)}^i\in A_i$ are uniquely determined. This provides the desired bijection $\phi$  and settles requirement (F1).

Next suppose that we are given a path $xyz$ with $x\in E_i, y\in E_j, z\in E_l$ for distinct colours $i,j,l$ in $\{1,\dots, k\}$. The key ingredient is that by condition (5) in Definition~\ref{defi:cubestructure}, 
each directed cube, in the sense of (F2), arises as a directed copy of one of the $3$-dimensional cubes of the complex. We now identify a directed cube satisfying the hypotheses of (F2) and an associated geometric $3$-cube.

First, there is a unique square $S_1^{ij}$ which contains a vertical-horizontal pair $a_r^i a_s^j$ with $a_r^i\in A_i, a_s^j\in A_j$ so that $x=\mathfrak{a}_r^i$ and $y=\mathfrak{a}_s^j$. Upon completing the square $S_1^{ij}$ to $a_r^i a_s^j=a_{s^1}^j a_{r^1}^i$, as in the beginning of the proof, for unique $r^1\in \{1,\dots, 2L_i\}$ and $s^1\in \{1,\dots, 2L_j\}$, we have
\[
xy\sim y^1x^1\text{ for }x^1=\mathfrak{a}_{r^1}^{i}\text{ and } y^1=\mathfrak{a}_{s^1}^j.
\]

Next we use $x^1$ and $z$ to extract a geometric square $S_2^{il}$, determined by $a_{r^1}^i a_t^l=a_{t^1}^l a_{r^2}^i$, for unique $t^1 \in \{1,\dots, 2L_l\}$ and $r^2\in \{1,\dots, 2L_i\}$, so that
\[
x^1z\sim z^1x^2 \text{ for } z=\mathfrak{a}_t^l, z^1=\mathfrak{a}_{t^1}^l\text{ and }x^2=\mathfrak{a}_{r^2}^i.
\]

Finally, by using $y^1,z^1$ we extract a geometric square $S_3^{jl}$, determined by
$a_{s^1}^ja_{t^1}^l =a_{t^2}^la_{s^2}^j $, for unique $s^2 \in \{1,\dots, 2L_j\}$ and $t^2\in \{1,\dots, 2L_l\}$, so that
\[
y^1z^1\sim z^2y^2 \text{ for }y^2=\mathfrak{a}_{s^2}^j\text{ and }z^2=\mathfrak{a}_{t^2}^l.
\]
There is a unique geometric cube containing the squares $S_1^{ij},S_2^{il}$ and $S_3^{jl}$, where the notation follows the convention after figure 1, and $a_r^i a_s^ja_t^l $ is a path joining two vertices in the cube at  longest possible distance due to condition (5) in Definition~\ref{defi:cubestructure}. Thus, in this geometric cube we have also obtained the path $a_{t^2}^l a_{s^2}^j  a_{r^2}^i$ opposite to $a_r^i a_s^ja_t^l $ and joining the same vertices in the  cube.

If we perform the same argument starting with $y,z$ to obtain $yz\sim z_1y_1$, followed by $x,z_1$ to obtain $xz_1\sim z_2x_1$ and finally $x_1,y_1$ to get $x_1y_1\sim y_2x_2$, we find unique squares $S_4^{jl}, S_5^{il}, S_6^{ij}$ determined by $a^j_sa^l_t=a^l_{t_1}a^j _{s_1}$, $a^i_ra^l_{t_1}=a^l_{t_2}a^i_{r_1}$ and $a^i_{r_1}a^j_{s_1}=a^j_{s_2}a^i_{r_2}$, respectively, for $r_2\in \{1,\dots, 2L_i\}, s_2\in \{1,\dots, 2L_j\}$ and $t_2\in \{1,\dots, 2L_l\}$. So
\[
x_2=\mathfrak{a}_{r_2}^i, \, y_2=\mathfrak{a}_{s_2}^j\, z_2=\mathfrak{a}_{t_2}^l.
\]
Since $a_r^i a_s^ja_t^l$ is a common $ijl$-path, the squares $S_1^{ij},S_2^{il}, S_3^{jl},S_4^{jl}, S_5^{il}, S_6^{ij}$ determine the same geometric cube. We have that $a_{t_2}^la_{s_2}^j  a_{r_2}^i$ is another path in this $3$-dimensional cube opposite to $a_r^i a_s^ja_t^l $ and joining vertices at longest possible distance. Since there only is one  path of longest distance opposite to $a_r^i a_s^ja_t^l $ in a $3$-dimensional cube,  we must have
\[
a_{s^2}^j=a_{s_2}^j,\, a_{t^2}^l =a_{t_2}^l \text{ and } a_{r^2}^i=a_{r_2}^i.
\]
Then  $x_2=x^2$, $y_2=y^2$ and $z_2=z^2$, as required to fulfill condition (F2).

We now assume that $\mathcal{X}$ has $N$ vertices, with $N\geq 2$, and we declare these to be the vertices of $\operatorname{DG}(\mathcal{X})$. As prescribed, each geometric edge in the complex is sent into two edges with opposite orientation in the edge set $E$ of $\operatorname{DG}(\mathcal{X})$. For every path of length two $xy$ with $x\in E_i$ and $y\in E_j$ for $i\neq j$ so that $o(y)=t(x)$, there is a unique geometric square $S_1$ in $\mathcal{X}$ in which $xy$ corresponds to a vertical-horizontal pair of edges. Then the corresponding horizontal-vertical pair of edges gives rise to $x'\in E_i$ and $y'\in E_j$ such that $xy\sim y'x'$, and this defines uniquely the bijection $\phi$ on the set $Y$ of paths of length two of distinct colours required in (F1). The condition (F2) holds by the same argument as the one-vertex case  because any $ijl$-coloured path $xyz$ will be contained in a unique $3$-cube in $\mathcal{X}$, determined through unique squares $S_1^{ij},S_2^{il}, S_3^{jl},S_4^{jl}, S_5^{il}, S_6^{ij}$. The difference is that since the complex has more than one vertex, we do not have a labelling of the edges by elements of the group $G$ acting cocompactly on $\mathcal{X}$, but this does not affect the existence of the squares and of the $3$-cube in which $xyz$ determines a path  with unique opposite path at longest distance between the same vertices of the $3$-dimensional cube.
\end{proof}

To visualise the argument in the proof of Theorem~\ref{thm:one vertex k-digraph from k-cube}, we refer to the generic geometric cube in figure~\ref{figure generic cube}. Let $x=a_1$ (or, for consistency, $x$ is a directed edge in $E_i$ labelled with $a_1\in A_i$), $y=b_1$ and $z=c_3$. The argument produces the path $a_{t^2}^la_{j^2}^j a_{r^2}^i$ given by $c_1b_3a_3$  following the squares $S_1^{ij}=(a_1,b_1,a_2^{-1}, b_2^{-1})$, $S_2^{il}=(a_2,c_3, a_3^{-1}, c_2^{-1})$ and $S_3^{jl}=(b_2, c_2, b_3^{-1}, c_1^{-1})$. Alternatively, it produces the path $a_{t_2}^la_{j_2}^j a_{r_2}^i$ following the  squares $(b_1, c_3, b_4^{-1}, c_4^{-1})$, $(a_1,c_4, a_4^{-1}, c_1^{-1})$ and $(a_4,b_4, a_3^{-1}, b_3^{-1})$.

\begin{example}\label{ex:figure1}
In figure~\ref{figure concrete geometric} we present a geometric cube which is part of the data of the $3$-cube group $\Gamma_2$ from Section \ref{subsection:RSV groups}.

\begin{figure}[h]
\centering
\includegraphics[height=4cm]{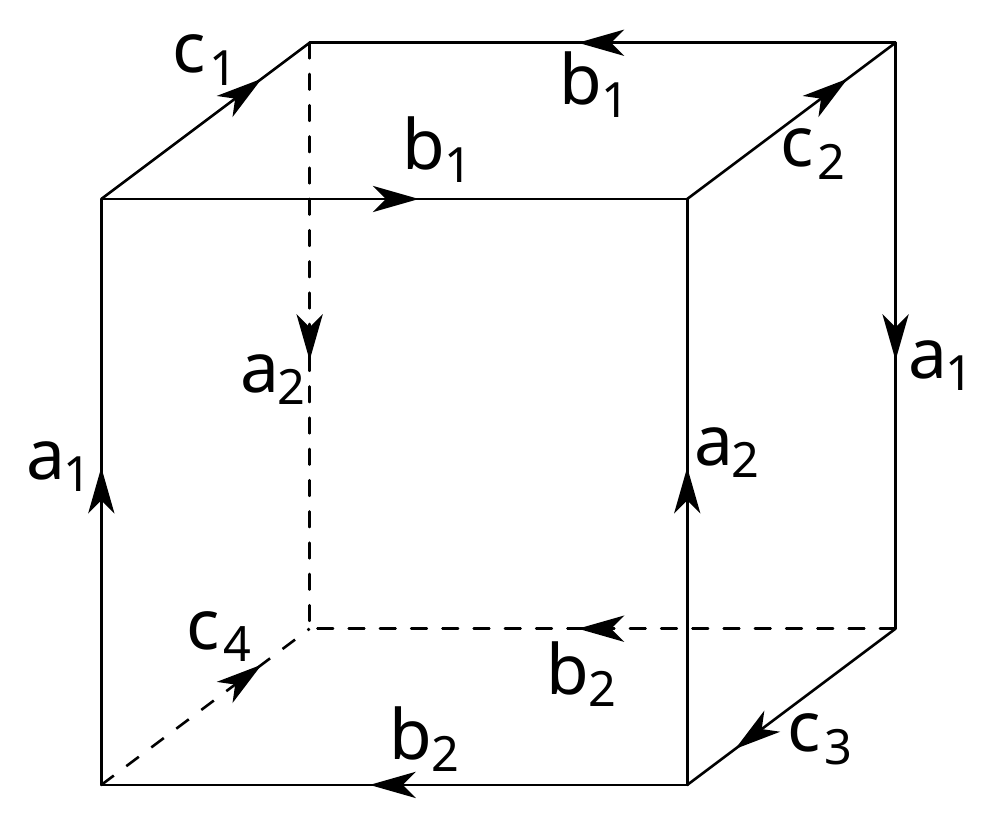}
\caption{A geometric cube for the $\Gamma_2$ group}
\label{figure concrete geometric}
\end{figure}

The generating sets of $\Gamma_2$ are $A_1=\{a_1,a_2, a_1^{-1},a_2^{-1}\}$, $A_2=\{b_1,b_2, b_3, b_1^{-1},b_2^{-1},b_3^{-1}\}$ and $A_3=\{c_1,c_2,c_3, c_4, c_1^{-1}, c_2^{-1},c_3^{-1},c_4^{-1}\}$. There are $(|A_1|\cdot |A_2|\cdot |A_3)|/2^3=24$ cubes in total,
where the factor $2^3$ in the denominator corresponds to the fact that there are $8$ vertices in the cube, and we can complete the cube starting with three edges of distinct colours from any one of them.

 The cube  in figure~\ref{figure concrete geometric} is obtained from the triple $(a_1, b_1, c_2)$ of edges in the three alphabets by completing its faces with geometric squares. With the notation of figure~\ref{figure generic cube}, the faces $S_1^{12}$, $S_2^{13}$ and $S_3^{23}$ arise, respectively, from the group relations
$a_1b_1a_4b_2$, $a_2c_2a_1c_3$ and $b_2c_4b_5c_3$ (identified with $b_2^{-1}c_3^{-1}b_2c_4^{-1}$). The remaining three faces correspond to the geometric squares $b_1c_2b_1c_5$, $a_1c_1a_2c_8$ (identified with $a_1c_1a_2c_4^{-1}$) and  $a_4b_4a_1b_2$ (identified with $a_2^{-1}b_1^{-1}a_1b_2$).
\end{example}

\subsection{Aperiodicity} In $C^*$-algebra theory, the classification of purely infinite, simple, unital, nuclear,  $C^*$-algebras is a landmark result by Kirchberg-Phillips, see \cite{Phi}.
Aperiodicity of a higher rank graph is an important property, because together with cofinality it implies simplicity of the associated $C^*$-algebra, and further implies pure infiniteness if every vertex can be reached from a loop with an entrance. We next investigate aperiodicity of $\Lambda(P)$ from Corollary~\ref{cor:the one vertex k-graph associated to a digraph}.

We recall the necessary facts and notation from \cite{KP}. Let $\Lambda$ be a $k$-graph. If ${m}=(m_i)_i, {q}=(q_i)_i\in \N^k$, we write  ${m}\leq {q}$ if $m_i\leq q_i$ for all $i=1,\dots,k$. By $\Omega_k$ we denote the $k$-graph with vertex set $\Omega_k^0=\N^k$ and set of elements (morphisms) consisting of pairs $(m,n)\in \N^k\times \N^k$ with $m\leq n$ and $d(m,n)=n-m$. The set $\Lambda^\infty$ of infinite paths consists of degree preserving functors $\omega:\Omega_k\to \Lambda$. An infinite path $\omega$ is \emph{aperiodic} provided that for every  $q\in \N^k$ and all $p\in \Z^k\setminus\{0\}$, there is $(m,n)\in \Omega_k$ such that $m+p\geq 0$ and $\omega(m+p+q,n+p+q)\neq \omega(m+p,n+p)$. The $k$-graph $\Lambda$ satisfies the \emph{aperiodicity condition (A)} provided that for every $v\in \Lambda^0$ there is an aperiodic path $\omega$ with $r(\omega)=v$.

In our case, the existence of an aperiodic infinite path will be provided by  the theory of rigid $k$-monoids from \cite{LV}. Guided by the work of Lawson-Vdovina, we first extend the notion of left and right rigid to $k$-dimensional digraphs with one vertex.  We note that the idea of rigidity below appeared in a first form in \cite[Definition on page 3, items(2) and (3)]{vdovina-pol}.

\begin{definition}\label{def:left right rigid digraph}
Let $\operatorname{DG}$ be a $k$-dimensional digraph with one vertex  and edge set $E=E_1\sqcup E_2\sqcup\dots\sqcup E_k$ for $k\geq 2$. We say that
\begin{enumerate}
\item $\operatorname{DG}$ is \emph{right rigid} if for every ${x}'\in E_i, {y}'\in E_j$ with $i\neq j$ there are unique $x\in E_{i}, y\in E_j$ such that $x{y}'\sim y{x}'$.
\item $\operatorname{DG}$ is \emph{left rigid} if for every $x\in E_i, y\in E_j$, $i\neq j$, there are unique ${x}'\in E_i, {y}'\in E_j$  such that $x{y}'\sim y{x}'$.
\end{enumerate}
\end{definition}

\begin{lemma}\label{lem:rigid} Suppose that $P$  is a one-vertex $k$-cube complex with underlying structure $(A_1,\dots,A_k)$ for  $i=1,\dots,k$, where each $A_i$ is of the form $\{a_1^i,\dots,a_{2L_{i}}^i\}$ and $a_r^i a_{L_i+r}^i=1$ in the associated group for all $1\leq r\leq L_i$. Let $\operatorname{DG}(P)$  be the  associated $k$-dimensional digraph from Theorem~\ref{thm:one vertex k-digraph from k-cube}. Then $\operatorname{DG}(P)$ is left and right rigid.
\end{lemma}

\begin{proof} The two properties of being rigid above arise from the fact that the link of the vertex $v$ in $P$ has no multiple edges. Therefore, every top-left corner and every bottom-right corner appear exactly once in a geometric square. The formal proof is below.

Suppose that $x'\in E_i$ and $y'\in E_j$ for  $i\neq j$ in $\{1,\dots,k\}$. Since we have a cube complex, there is an associated square complex $P_{A_i, A_j}$. By our construction of $\operatorname{DG}(P)$, there are  $a_r^i\in A_i$ and $b_s^j\in A_j$ for unique $r=1,\dots, 2L_i$ and $s=1,\dots,  2L_j$ such that $x'=\mathfrak{a}_r^i$ and $y'=\mathfrak{b}_s^j$.
In the square complex $P_{A_i, A_j}$ there is a unique square of the form
\[
S_O=(a_r^i, (b_s^j)^{-1}, (a_g^i)^{-1}, b_h^j)
\]
for $g\in \{1,\dots, 2L_i\}$ and $h\in \{1, \dots,  2L_j\}$. The associated $S_H$ is $(a_g^i, b_s^j, (a_r^i)^{-1}, (b_h^j)^{-1})$, and thus $a_g^i b_s^j=b_h^j a_r^i$ in $G_{A_i, A_j}$. Letting $x=\mathfrak{a}_g^i$ and $y=\mathfrak{b}_h^j$ gives $xy'\sim yx'$ in $\operatorname{DG}(P)$, as claimed.

Right rigidity  is similar. Starting this time with $x\in E_i$ and $y\in E_j$ for distinct $i,j$ in $\{1,\dots,k\}$, we find $a_r^i\in A_i$ and $b_s^j\in A_j$ for unique $r\in \{1,\dots, 2L_i\}$ and $s\in \{1,\dots,  2L_j\}$ so that $x=\mathfrak{a}_r^i$ and $y=\mathfrak{b}_s^j$. Consider the unique square
 \[
S_O=((a_r^i)^{-1},b_s^j, a_m^i, (b_n^j)^{-1})
\]
in  $P_{A_i, A_j}$, and form its associated $S_V$, which is $(a_r^i, b_n^j, (a_m^i)^{-1}, (b_s^j)^{-1})$. Then  $a_r^i b_n^j=b_s^j a_m^i$ in $G_{A_i, A_j}$, so letting $x'=\mathfrak{a}_m^i$ and $y'=\mathfrak{b}_n^j$ leads to $xy'\sim yx'$ in the digraph, as claimed.

\end{proof}

Given a one-vertex $k$-cube complex $P$, the associated $k$-graph $\Lambda(P)$ is a monoid, being a category with a single object. It is a $k$-monoid in the sense of \cite{LV}, with alphabets $E_i=\{\mathfrak{a}_1^i,\dots, \mathfrak{a}_{2L_i}^i\}$ for $i=1,\dots, k$ where each $L_i\geq 1$.

\begin{corollary}\label{cor:aperiodicity} Given a one-vertex $k$-cube complex $P$, the graph $\Lambda(P)$ is left and right rigid, and satisfies the aperiodicity condition. In particular, $C^*(\Lambda(P))$ is simple.
\end{corollary}

\begin{proof} Let $P$ be a one-vertex $k$-complex, which we may assume as in the hypothesis of Lemma~\ref{lem:rigid}. Let $\Lambda(P)=\Lambda(\operatorname{DG}(P))$ be the associated $k$-graph from Corollary~\ref{cor:the one vertex k-graph associated to a digraph}.  The right rigidity of the digraph implies that for any choice of elements $y',x'$ with $y'$ in the alphabet $E_j$ and $x'$ in $E_i$, where $i\neq j$, there are unique elements $x$ and $y$ in the alphabets $E_i$ and $E_j$, respectively, so that $y'\circ x=x'\circ y$. This means that $\Lambda(P)$ is right rigid. Left rigid follows in a similar way. We conclude  from \cite[Corollary 11.10 and Lemma 4.15]{LV} that $\Lambda(P)$ is effective and hence admits an aperiodic infinite path. As there is only one vertex, the aperiodicity condition is satisfied. Since $\Lambda(P)$ is also cofinal, $C^*(\Lambda(P))$ is simple by \cite[Proposition 4.8]{KP}.
\end{proof}

The constructions of \cite{RS} produce a purely infinite simple rank two Cuntz-Krieger algebra $\mathcal{A}$. This uses in a crucial way the fact that every word $w$ of a given shape  $m=(m_1,m_2)\in \N^2$ admits at least two distinct extensions $w',w''$, in the sense that the origin of $w',w''$ (with suitable interpretation) equals the terminus of $w$, and both have same  shape $e_j$ for all $j=1,2$.

For a row-finite and source free $k$-graph $\Lambda$, \cite[Proposition 4.9]{KP} put forward  conditions that would imply $C^*(\Lambda)$ is purely infinite simple. A correct version of these conditions was identified in \cite[Proposition 8.8]{Sim}, which we present here (writing \emph{cycle} instead of \emph{loop}): given a finitely aligned $k$-graph $\Lambda$, a morphism $\mu\in \Lambda\setminus \Lambda^0$  is a \emph{cycle with an entrance} if $s(\mu)=r(\mu)$ and there exists $\alpha\in s(\mu)\Lambda$ having $d(\alpha)\leq d(\mu)$ and being distinct from the initial segment of $\mu$ of degree $d(\alpha)$. Thus for some factorisation $\mu=\mu_1\mu_2$ where $d(\mu_1)=n\leq d(\mu)$,  there exists $\alpha\neq\mu_1$ with $d(\alpha)=n$ and  $r(\alpha)=r(\mu_1)$. Therefore, upon interpreting concatenation of edges on the digraph $\operatorname{DG}(P)$ as composition of morphisms in the associated $\Lambda(P)$, see \eqref{eq: formal composition of morphisms on edges}, and by interpreting the constructions of \cite{RS} in terms of higher rank graphs, the existence of a cycle with an entry requires that for a given $\mu_2$  there are
 two distinct extensions, with the additional property that the origin of $\mu_2$ is the terminus of one of the extensions. As we will show below, our $k$-graphs satisfy the stronger aperiodicity condition used in \cite{RS}.

\begin{proposition}\label{prop:purely infinite}
Let $P$ be a one-vertex $k$-complex as in the hypothesis of Lemma~\ref{lem:rigid} for $k\geq 2$ and let  $\Lambda(P)$ be the associated one-vertex $k$-graph.  {{If $L_i\geq 2$ for $i=1,\dots,k$, then the vertex in $\Lambda(P)$ supports at least two distinct cycles of length two in colour $i$, hence $C^*(\Lambda(P))$ is purely infinite.}} Furthermore, $C^*(\Lambda(P))$ falls under the Kirchberg-Phillips classification theory and is thus determined by its K-theory.
\end{proposition}

\begin{proof}
 Fix $i\in \{1,\dots,k\}$ with $L_i\geq 2$. Then we can form the length-two cycles $\mu=\mathfrak{a}_1^i\mathfrak{a}_{L_i+1}^i$  and $\nu=\mathfrak{a}_2^i\mathfrak{a}_{L_i+2}^i$ based at $v$ with $d(\mu)_i=2=d(\nu)_i$.{{ Now $\mathfrak{a}_2^i$ provides an edge $\alpha$ with nontrivial degree $d(\alpha)\leq d(\mu)$ which is an entry to $\mu$ not already contained in $\mu$}}. We conclude that $C^*(\Lambda(P))$ is purely infinite. The last claim follows by \cite[Corollary 8.15]{Sim}, see also \cite[Remark 5.2]{E}, by appealing to the classification result in \cite{Phi}.
\end{proof}

\section{Construction of $k$-graphs with several vertices}\label{sec:new examples}

In this section we present our construction of $k$-graphs with several vertices, for $k\geq 2$, and provide examples and applications.

Towards this aim we need a procedure to get $k$-cube complexes with several vertices. It is known  that in a complex with several vertices one cannot consistently identify the label of an edge with the label as generator in the fundamental group. We come around this challenge by introducing additional layers of labels, corresponding to covers with $N$ sheets, similar to what is done for $N=2$ in \cite[Section 8.1]{LSV}.
 In general, this is a hard problem, since there exist complexes without non-trivial finite covers, cf. \cite{burger-mozes:lattices}. Even when the complexes are known to admit $N$-covers, corresponding to subgroups of index $N$ of the fundamental group, see e.g. \cite[Theorem 1.38]{Hatcher},  it is difficult to construct covers explicitly. One challenge is that the subgroups can be defined in many different ways. For us a cover will be defined by picture, meaning that it is explicitly defined by the images of vertices, edges, faces and so on.  In all our pictures, the covering map amounts to forgetting the upper indexes, and we can explicitly see that we have a local homeomorphism at each point.

 Recall that in a one-vertex $k$-cube complex with generating structure $A_1,\dots, A_k$ we view edges as being coloured, with each $A_i$ for $i=1,\dots,k$ endowed with a distinct colour. For a complex $\mathcal{X}=\tilde{P}$ obtained as an $N$-cover of a one-vertex $k$-cube complex $P_{A_1,\dots,A_k}$, the coloured edges are given by $p^{-1}(A_i)$ for $i=1,\dots ,k$, where $p:\mathcal{X}=\tilde{P}\to P_{A_1,\dots,A_k}$ is the covering map. A $k$-complex, when viewed as undirected, is always connected.

\begin{proposition}\label{prop: double cover og k cube group}
Suppose that $G$ is a $k$-cube group with associated $k$-cube complex $P_{A_1,\dots,A_k}$. Then $P_{A_1,\dots,A_k}$ admits a double cover $p:\tilde{P}\to P_{A_1,\dots,A_k}$ with $\tilde{P}$ a complex with $2$ vertices.
\end{proposition}

We establish this proposition by writing down an {explicit} double cover, which will be prescribed "by picture" on $2$-cells and $3$-cells of the complexes under consideration, see Lemmas ~\ref{def: double cover of square} and \ref{lem:2-vertex 3-complex}. Before presenting the proof we
 point out a consequence.

\begin{corollary}
Each  $k$-cube group $G$ admits a subgroup of index $2$.
\end{corollary}

\begin{proof} By e.g. \cite[Theorem 1.38]{Hatcher}, for a given path-connected, locally path-connected, and semilocally simply connected space $X$ there is a bijection between the set of basepoint preserving isomorphism classes of path-connected covering spaces $p:(\tilde{X},\tilde{x_0})\to (X,x_0)$ and the set of subgroups of $\pi_1(X,x_0)$. The correspondence associates the subgroup $p_\ast(\pi_1(\tilde{X},\tilde{x_0}))$ to the covering $(\tilde{X},\tilde{x_0})$, and the number of sheets of the covering equals the index
of $p_\ast(\pi_1(\tilde{X}, \tilde{x_0}))$ in $\pi_1(X,x_0)$, see   \cite[Proposition 1.32]{Hatcher}. Applying this to the $2$-cover from Proposition~\ref{prop: double cover og k cube group} yields existence of a subgroup of $G$ of index $2$.
\end{proof}

In order to motivate our constructions of coverings, we review a construction of a $2$-cover of the complex associated to the torus $\mathbb{T}^2$.

\begin{figure}[h]
\centering
\includegraphics[height=3cm]{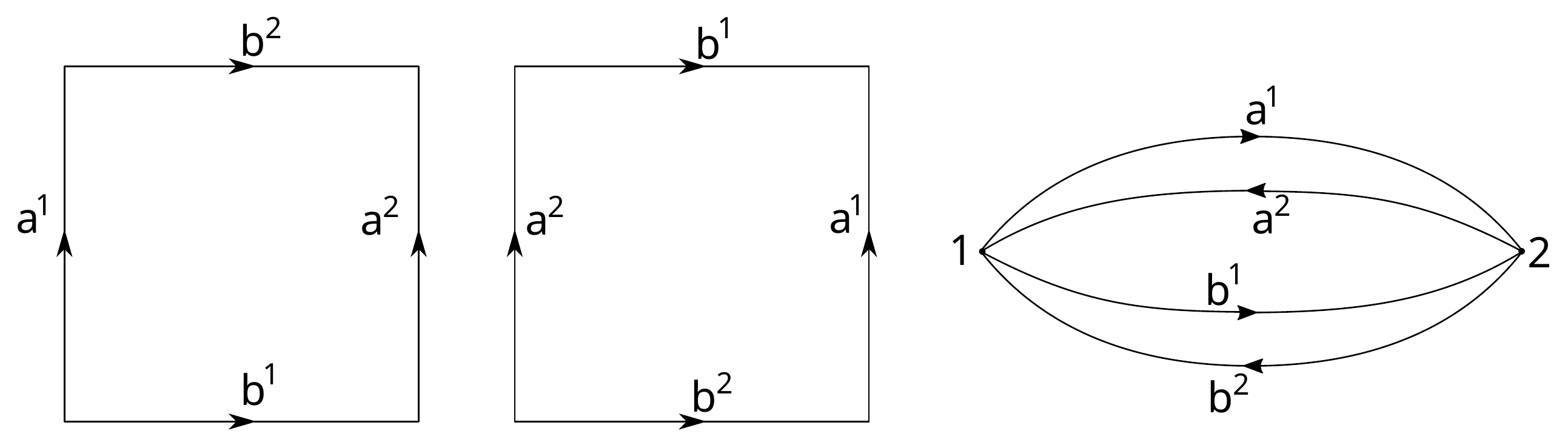}
\caption{A double cover of the torus}
\label{figure double cover of torus}
\end{figure}

Recall from, for example \cite[Page 14]{Hatcher}, that  $\mathbb{T}^2$ is obtained from a $2$-cell given by a square with pairs of opposite edges having same orientation and label $a$ (vertically) or $b$ (horizontally), by gluing it onto the wedge of two circles. A double cover with \emph{two vertices} arises from two squares with oriented edges having distinct labels $a^1$, $a^2$, $b^1$ and $b^2$, with the upper index $1$ or $2$ indicating the source vertex, as shown in figure~\ref{figure double cover of torus}.  The complex is obtained by attaching the two squares (the $2$-cells)  to the graph with vertices 1 and 2 in Figure~\ref{figure double cover of torus}. If $P$ is the complex associated to the $(2,2)$-group in Example~\ref{ex:complex from torus}, then the cover $\tilde{P}$ just described will lead to a $2$-graph with two vertices, see Theorem~\ref{thm:N-vertex k graph}, see also Figure 2 in \cite[Section 8.1]{LSV}.

The proof of Proposition~\ref{prop: double cover og k cube group} relies on a two lemmas, with the first one detailing an explicit double cover for the $2$-cells of the given $k$-cube complex.

\begin{lemma}\label{def: double cover of square}
Suppose that we have a one-vertex square complex $S=S_{A,B}$ with VH-structure $(A, B)$ and vertex $v$. Then there is a $2$-cover $p:\tilde{S}\to S$ given by a square complex $\tilde{S}$ with two vertices $v_1$ and $v_2$ whose squares are given by the prescription: The inverse image of a geometric square $S_O^{a,b}=(a,b,c^{-1},d^{-1})$ in $S_{A,B}$ consists of two geometric squares, $S_1=(a^1,b^2, (c^2)^{-1}, (d^{1})^{-1})$ and $S_2=(a^2,b^1, (c^1)^{-1}, (d^{2})^{-1})$ in $\tilde{S}$, and the covering map is determined by
\[
p(*^1)=p(*^2)=*\text{ for }*=a,b,c,d.
\]
\end{lemma}
Equivalently, if we denote $v_1=1$ and $v_2=2$, the covering $p$ is depicted on the squares by
\begin{equation}\label{eq:explicit two cover of squares}
\def\g#1{\save[]!C="g#1"\restore}%
\xymatrix{
\g1 \ar @{}[dr] |{S_1}{\color{red}{2}}\ar[r]^{b^2} & {1} &\g2 \ar @{}[dr] |{S_2}{1}\ar[r]^{b^1} & {\color{red}{2}} &\g3 {} &\g4 \ar @{}[dr] |{S_O^{a,b}} {1}\ar[r]^{b} &{1}\\
{1}\ar[u]^{a^1}\ar[r]_{d^1} & {\color{red}{2}}\ar[u]_{c^2} & {\color{red}{2}}\ar[u]^{a^2}\ar[r]_{d^2} & {1}\ar[u]_{c^1} &{} & {1}\ar[u]^{a}\ar[r]_{d} & {1}\ar[u]_{c}
}
\end{equation}

Note that in the $2$-cover there are two geometric edges for each geometric edge in $S_{A,B}$, so for example to $a\in S_{A,B}$ having origin and terminus the vertex $1$ (identified with $v$) there will correspond $a^1$ and $a^2$ in $\tilde{S}$, with $a^1$ having origin $1$ and terminus $2$, and $a^2$ with origin $2$ and terminus $1$.

\begin{proof}
We need only observe that $p(S_*^{a^1,b^2})=S_*^{a,b}=p(S_*^{a^2, b^1})$ for $* =O, H, V, R$. Thus the square complex $\tilde{S}$ is well-defined. The map $p$ is a local homeomorphism because it is defined on the cells, and sends edges to edges and vertices to vertices.
\end{proof}

\begin{lemma}\label{lem:2-vertex 3-complex}
Let $P_{A_1,\dots,A_k}$ be a $k$-cube complex associated with a $k$-cube group $G$ with underlying structure determined by the ordered tuple $(A_1,\dots,A_k)$, with $\#A_i=2L_i$ for every  $i=1,\dots,k$.
There is a $2$-cover $\tilde{P}$ of $P_{A_1,\dots,A_k}$ determined as follows: On each $2$-dimensional cell, the covering $p$ is defined in Lemma~\ref{def: double cover of square}. On a $3$-dimensional geometric cube $C$, such as is described in figure~\ref{figure generic cube} where we assume $a_r\in A_i$, $b_s\in A_j$ and $c_t\in A_l$, for $r,s,t=1,\dots, 4$ and $i,j,l\in \{1,\dots,k\}$, the cover $\tilde{C}$ of $C$ consists of two geometric cubes, see figure~\ref{figure two cubes in a two cover}, with labelling of edges $\{a_r^\epsilon\}$, $\{b_s^\epsilon\}$ and $\{c_t^\epsilon\}$ for $\epsilon=1,2$, and with the covering map given by
\begin{equation}\label{eq:map p for cubes}
p:\tilde{C}\to C, p(a_r^1)=p(a_r^2)=a_r, r=1,\dots, 4,
\end{equation}
and similarly for $p(b_s^\epsilon)$ and $p(c_t^\epsilon)$. For $4\leq l\leq k$, the map $p$ is defined on an arbitrary $l$-cube by its prescription on the underlying $3$-cubes.
\end{lemma}

\begin{figure}[h]
\centering
\includegraphics[height=4cm]{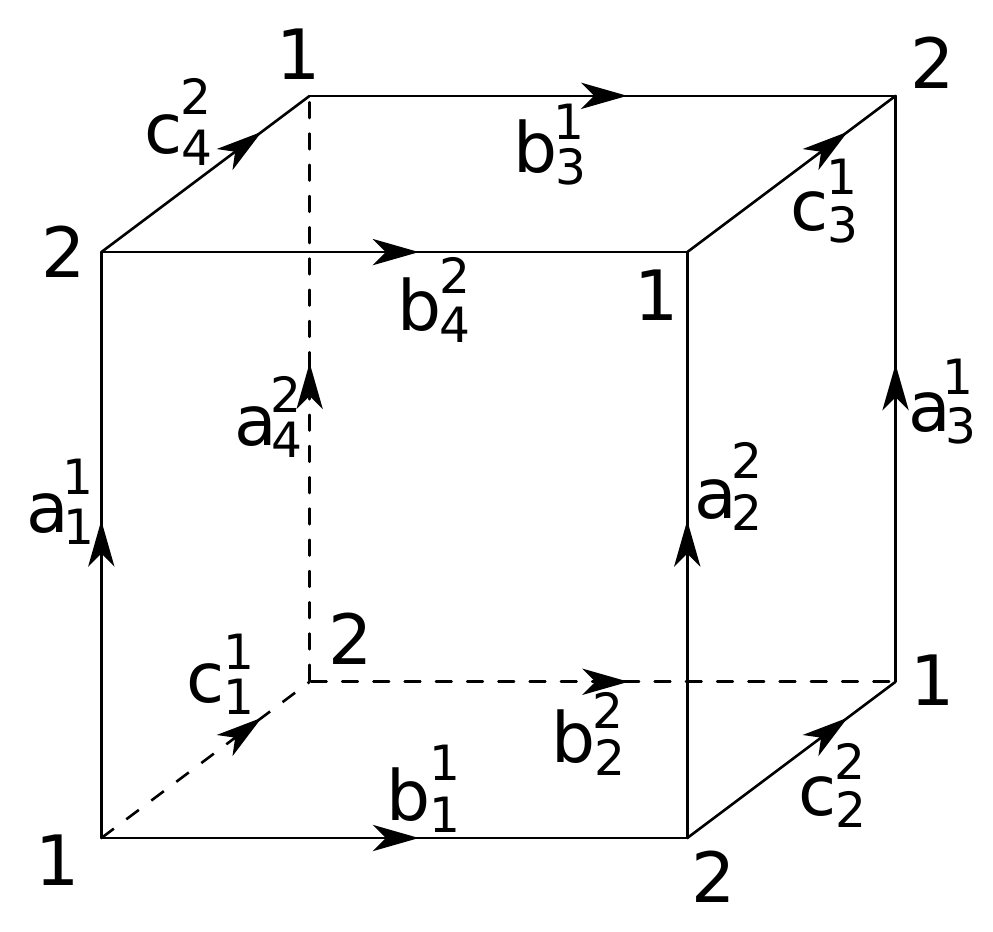}\qquad
\includegraphics[height=4cm]{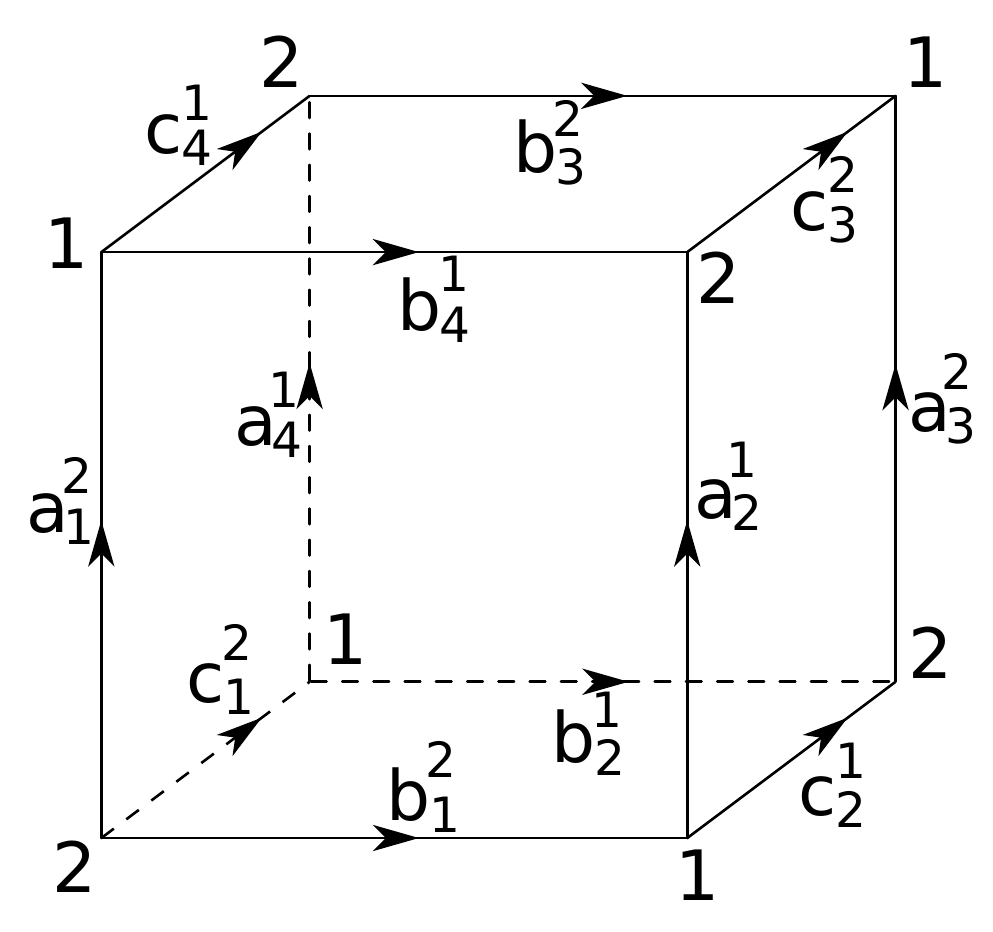}
\caption{A two cover of a generic geometric cube}
\label{figure two cubes in a two cover}
\end{figure}

\begin{proof} At $k=3$, it suffices to verify that $p$ on $\tilde{C}$ is well-defined. But this is clear  from the construction of the map $p$ in \eqref{eq:map p for cubes}. The map $p$ is constructed recursively on higher dimensional cubes: if $C$ is an $l$-cube for $4\leq l\leq k$, then  $p$ is prescribed  consistently on all $(l-1)$-dimensional faces of $C$, similarly to how \eqref{eq:map p for cubes} is obtained on $3$-cubes from its prescription in \eqref{eq:explicit two cover of squares} on squares.
\end{proof}

\begin{proof}[Proof of Proposition~\ref{prop: double cover og k cube group}] This follows by applying Lemmas \ref{def: double cover of square} and \ref{lem:2-vertex 3-complex}.
\end{proof}

To obtain $k$-complexes with $N$ vertices for $N>2$, there are several ways to use $k$-cube groups.
Many of the $k$-cube groups are residually finite, so, because of the 1-to-1 correspondence between the subgroups of index $N$ and $N$-covers of the corresponding $k$-cube complex, we can get infinitely many $k$-cube complexes with $N$ vertices. {In principle, different  subgroups of the same index $N$ can lead to different coverings. If $N\geq 3$, then the labelling of vertices is hard to sort out and we do not know of an explicit prescription similar to \eqref{eq:map p for cubes}. We shall use the following general procedure.

\begin{proposition}\label{prop:N vertex complex from symmetric group}
 Suppose that $G$ is a residually finite $k$-cube group with $k\geq 2$. To each normal subgroup $H$ of $G$ of finite index $N$ there is a $k$-complex $\mathcal{X}$ with $N$ vertices obtained by the following prescription: let $Q:G\to S_N$ the homomorphism obtained by composing the embedding of $G/H$ into the symmetric group on $N$ letters $S_N$ given by Cayley's theorem with the quotient map $q:G\to G/H$. For $a\in G$, the permutation $Q(a)$ in $S_N$ encodes the edges in the complex, with $a^n$ labelling an edge from the vertex $n$ to the vertex $n'=Q(a)(n)$ for $n=1,\dots,N$.

\end{proposition}

\begin{proof}
Let $P$ be the one-vertex complex determined by $G$. Suppose that $H$ is a subgroup of $G$ so that $(G:H)=N$. The complex $\mathcal{X}$ is constructed by associating to each square $S_O=(a,b,(a')^{-1}, (b')^{-1})$ in $P$  a total of $N$ squares
\[(Q(a), Q(b), Q((a')^{-1}), Q((b')^{-1}))\] in the new, $N$-vertex complex, with vertices labelled by elements in $\{1,2,\dots,N\}$.
 Indeed, applying $Q$ to the relation $ab=b'a'$ in $G$ gives the identity $Q(a)Q(b)=Q(b')Q(a')$ in $S_N$, which in turn yields $N$ squares in the complex $\mathcal{X}$ determined by
\[
Q(a)Q(b)(n)=Q(b')Q(a')(n), \text{ for all }n=1,\dots, N.
\]
More precisely, for each $n$, let $m=Q(a)Q(b)(n)=Q(b')Q(a')(n)$ and consider the square with labelling $(a')^n(b')^s$ in vertical-horizontal direction from vertex $n$ to vertex $m$ via vertex $s$, and with labelling $b^na^r$ in horizontal-vertical direction via vertex $r$, where $m=Q(a)(r)$, $r=Q(b)(n)$, $Q(a')(n)=s$ and $Q(b')(s)=m$, see the figure, where $1$ in the right square denotes the vertex in the complex of $G$:

\begin{equation}\label{eq:explicit N cover of squares}
\def\g#1{\save[]!C="g#1"\restore}%
\xymatrix{
\g1 \ar @{}[dr] |{S}{s}\ar[r]^{(b')^s} & {m} &\g2 \ar @{}[dr]
 |{S_O^{a,b}} {1}\ar[r]^{b} &{1}\\
{n}\ar[r]_{b^n}\ar[u]^{(a')^n} & {r}\ar[u]_{a^r} & {1}\ar[u]^{a}\ar[r]_{b'} & {1}\ar[u]_{a'}
}
\end{equation}
The covering map $p:\mathcal{X}\to P$ is defined by collapsing all squares of the form $S$ onto the given square $S_O^{a,b}$ in $P$, for each square in $P$.
\end{proof}

We extend the notions of left and right rigid  to $k$-dimensional digraphs  and $k$-graphs with more than one vertex in the natural way. The idea is that being rigid means that if two edges can form a corner (either bottom-left or top-right), then they do form a unique corner.

\begin{definition}\label{def:rigid for multiple vertices}
(a) Suppose that $\operatorname{DG}$ is a $k$-dimensional digraph. Then $\operatorname{DG}$ is right rigid if for $x\in E_i$ and $y\in E_j$ edges of distinct colours $i\neq j$ so that  $o(x)=o(y)$,  there are unique $x'\in E_i$ and $y'\in E_j$ with  $xy'\sim yx'$. Left rigid is defined in a similar way.

(b) Suppose that $(\Lambda,d)$ is a $k$-graph, and let $E_i=d^{-1}(\{e_i\})$ be its alphabets, for $1\leq i\leq k$. We say that $\Lambda$ is right rigid if for $x\in E_i$ and $y\in E_j$ with the same origin, for $i\neq j$, there are unique $x'\in E_i$ and $y'\in E_j$ with the property that $y'\circ x=x'\circ y$. Left rigidity of $\Lambda$ is defined in a similar manner.
\end{definition}

\begin{theorem}\label{thm:N-vertex k graph}
Suppose that $P$ is a one-vertex $k$-cube complex that admits an $N$-cover $p:\tilde{P}\to P$ as in Proposition~\ref{prop:N vertex complex from symmetric group},  with $\mathcal{X}=\tilde{P}$  the associated $k$-cube complex with $N$ vertices,  for $k\geq 2$ and $N\geq 2$. The following are valid.

(a) The $k$-dimensional digraph $\operatorname{DG}(\mathcal{X})$ determined in Theorem~\ref{thm:one vertex k-digraph from k-cube} is left and right rigid.

(b) The $k$-graph $\Lambda({\mathcal{X}}):=\Lambda(\operatorname{DG}({\mathcal{X}}))$ associated to $\operatorname{DG}(\mathcal{X})$ by Corollary~\ref{cor:the one vertex k-graph associated to a digraph} is strongly connected, left rigid and right rigid.
\end{theorem}

\begin{proof} Part (a) follows from  Theorem~\ref{thm:one vertex k-digraph from k-cube}. Turning to part (b), to show  that $\Lambda(\mathcal{X})$ is strongly connected let $v,w$ be distinct vertices. By the construction of cover in Proposition~\ref{prop:N vertex complex from symmetric group}, there is an element of the $k$-cube group whose action on $v$ gives $w$. Associated to this element there is a path $y_1y_2\dots y_m$ in the $1$-skeleton of the $k$-complex $P$, and this has a unique lift to a path in $\mathcal{X}$  from $v=o(y_1)$ to $w=t(y_m)$.  In particular, $v_s=t(y_s),v_{s+1}=o(y_{s})$ are adjacent vertices in the complex for each $1\leq s\leq m-1$ (identifying $v=v_1$). Our definition of the $k$-dimensional digraph gives two directed edges, with opposite orientation, having source $v_s$ and terminus $v_{s+1}$, respectively the opposite, for each $s=1,\dots,m-1$. This allows to form directed paths in $\operatorname{DG}(\mathcal{X})$, hence in $\Lambda(\mathcal{X})$, from $v$ to $w$ and from $w$ to $v$, as needed.

To see that $\Lambda(\mathcal{X})$ is rigid, it suffices to note that every vertex in the cover $\tilde{P}$ has the same link as the one vertex of $P$, and in particular its link contains no multiple edges. Therefore the proof of Lemma~\ref{lem:rigid} carries through.
\end{proof}

\begin{corollary}\label{cor:aperiodic}
The $k$-graph $\Lambda(\mathcal{X})$ from Theorem~\ref{thm:N-vertex k graph}  satisfies the aperiodicity condition.
\end{corollary}

\begin{proof}
Since $\Lambda(\mathcal{X})$ is rigid, the existence of an aperiodic path in $\Lambda(\mathcal{X})$ based at a given vertex is guaranteed as in the one-vertex case, see \cite[Lemma 4.15 and Corollary 11.10]{LV}.
\end{proof}

\begin{proposition}\label{prop:purely infinite multiple vertices}
Assume the hypotheses of Theorem~\ref{thm:N-vertex k graph}, where $P=P_{A_1,\dots,A_k}$ with $A_i$ given as in Lemma~\ref{lem:rigid} for $i=1,\dots,k$. If $|A_i|\geq 2$ for all $i=1,\dots,k$, then every vertex $\Lambda(\mathcal{X})$ supports at least two cycles of each colour. In particular,
$C^*(\Lambda(\mathcal{X}))$ is purely infinite and therefore classifiable by the Kirchberg-Phillips classification theory.
\end{proposition}

\begin{proof}
Let  $\Lambda(\mathcal{X})^0$ denote the vertices, or identities, in our $k$-graph. Since every geometric edge in $P$ gives rise to $N$ geometric edges in $\tilde{P}$, we have that for each colour $i\in \{1,\dots,k\}$, every vertex $v\in \Lambda(\mathcal{X})^0$ admits $N|A_i|$ incident edges, namely edges with origin or terminus $v$. Furthermore, by our construction of $\Lambda(\mathcal{X})$ we also know that
each edge is contained in a length-two cycle. Thus, for a given $v\in \Lambda(\mathcal{X})^0$, there are at least two cycles $\mu=x_2\circ x_1$ and $\nu=x_4\circ x_3$ based at $v$ and consisting of edges of colour $i$, with the terminus of $x_1$ possibly distinct from the terminus of $x_3$. Then $x_4$ is an entry to $\mu$ of smaller degree and not already contained in $\mu$. In this consideration the vertex $v$ already supports a cycle with an entrance, but since our $k$-graph is strongly connected we could have chosen a cycle $\mu$ based at a different vertex $w$ and apply the same consideration.
Now  \cite[Proposition 8.8  and Corollary 8.15]{Sim} apply to give the claimed conclusion.
\end{proof}

We next illustrate our construction of $k$-graphs with more than one vertex with an explicit example of an infinite family of $k$-graphs with two vertices, for all $k\geq 2$. The construction was partly outlined in \cite[Section 8.1]{LSV}, as corresponding to the uniform labelling $l_u$, and explicit factorisation rules of the $2$-vertex graph were given in the case of the mixed labelling $l_m$. Here we describe completely the case $l_u$ as an application of Theorem~\ref{thm:N-vertex k graph}.

\begin{proposition}\label{prop:construct k graphs with arbitrary spectral radii}
For $k\geq 2$ and any $k$-tuple $(L_1,\dots,L_k)$ of positive integers there exists an aperiodic strongly connected $2$-vertex $k$-rank graph $\Lambda$ with $\vert v\Lambda^{e_i}w \vert= 2L_i$, where $v$ and $w$ are the vertices in $\Lambda$ and $i= 1,\dots,k$.
\end{proposition}

\begin{proof}
 Fix $k\geq 2$ and for each $i=1,\dots,k$ let $\mathcal{L}_i$ be an alphabet with $L_i$ letters.
 Let $\mathbb{F}_i=\mathbb{F}_{L_i}$ be the free group generated by $\mathcal{L}_i$ for each $i=1,\dots,k$. The product group $\mathbb{F}_{1}\times \cdots \times \mathbb{F}_{k}$ acts simply and transitively on the product of trees $\mathcal{T}_{2L_1}\times \dots\times\mathcal{T}_{2L_k}$, and yields in the quotient a complex $P$ with one vertex and skeleton a wedge of $\sum_{i=1}^k L_i$ circles. The $2$-cells in $P$ arise from pairs
 $a\in \mathcal{L}_i$, $b\in \mathcal{L}_j$ for $i\neq j$ with the commutation relation $ab=ba$ as in Example~\ref{ex:complex from torus}; for each such pair, there is a torus glued to the wedge of circles.

Let $\tilde{P}$ be the associated $2$-cover from Proposition~\ref{prop: double cover og k cube group}. By Theorem~\ref{thm:N-vertex k graph} and Corollary~\ref{cor:aperiodic} there is a $k$-graph $\Lambda:=\Lambda(\tilde{P})$ with the desired property:  for each geometric edge in $P$, say  having label $a\in \mathcal{L}_i$, there are two geometric edges labelled $a^1$ and $a^2$ in $\tilde{P}$, and each of these gives exactly one edge in the associated $k$-graph $\Lambda$ between the two vertices.
\end{proof}

\begin{example} To illustrate Proposition~\ref{prop:purely infinite multiple vertices} and Proposition~\ref{prop:construct k graphs with arbitrary spectral radii}, suppose that $k=2$ and $L_1=L_2=1$. The associated $\Lambda$ has $1$-skeleton the graph with two vertices in figure 6, where we identify $v$ as vertex $1$ and $w$ as vertex $2$. If we view the coloured edges in direction $e_1\in \N^2$ as labelled by $\mathcal{L}_1$ and in direction $e_2\in \N^2$ to be labelled by $\mathcal{L}_2$, then by Proposition~\ref{prop:construct k graphs with arbitrary spectral radii} we have
\[
w\Lambda^{e_1}v=\{a^1,\bar{a^2}\},\, w\Lambda^{e_2}v=\{b^1,\bar{b^2}\},\, v\Lambda^{e_1}w=\{a^2,\bar{a^1}\} \text{ and } v\Lambda^{e_1}w=\{b^2,\bar{b^1}\}.
\] The 8 factorisation rules are:
\begin{align*}
a^1b^2=b^1a^2,&a^2\bar{b^2}=\bar{b^1}a^1,\bar{a^1}b^1=b^2\bar{a^2},\bar{a^2}\bar{b^2}=\bar{b^1}\bar{a^1}\\
a^2b^1=b^2a^1,&a^1\bar{b^1}=\bar{b^2}a^2, \bar{a^2}b^2=b^1\bar{a^1},\bar{a^1}\bar{b^1}=\bar{b^2}\bar{a^2}.
\end{align*}
If $k=2$ and $L_1=L_2=2$, then the corresponding $2$-graph on two vertices has the same 1-skeleton, and for example  $|v\Lambda^{e_1}w|=|w\Lambda^{e_1}v|=4$, and similarly in colour $e_2$.
\end{example}

\section{Applications}

\subsection{Von Neumann algebras from strongly connected $k$-graphs}
We now present a large supply of von Neumann type $\mathrm{III}_\lambda$ factors from $k$-graphs as in \cite{LLNSW}, for infinitely many values of $\lambda$ in $(0,1]$. We start with some preparation.

We refer to \cite[Section 6]{RSV} for the notion of adjacency operator in $i$-direction for a $k$-cube complex, where $i\in \{1,\dots,k\}$. The basic ingredients are as follows: Let $X$ be a $k$-cube complex with vertex set (of its $1$-skeleton) denoted $X_0$ and with universal cover a product $\RT_1\times\RT_2\times\dots \times\RT_k$ of regular trees.  For each $i=1,\dots ,k$ and $V,W\in X_0$, we write  $V \sim_i W$ if the two vertices in the complex  are adjacent in the $i$-direction of $X$. The \textbf{adjacency operator $\mathcal{A}_i$ in $i$-direction} is defined  on $L^2(X_0) $ by
\[
\mathcal{A}_i(f)(V) = \sum_{W  \sim_i V} f(W).
\]
Since all complexes considered here are locally finite in a strong sense, meaning that at every vertex there are finitely many edges in each direction $i$, or of each colour $i$, for $i\in \{1,\dots,k\}$, the operators $\mathcal{A}_i$ become $|X_0|\times |X_0|$ matrices. It was further observed in \cite[Remark 6.4]{RSV} that whenever each pair of edges starting at a vertex of $X$ in direction ${i,j}$, with $i\neq j$, belong to a unique square in $X$, then $\mathcal{A}_i$ and $\mathcal{A}_j$ commute.

\begin{proposition}\label{prop:spectral radii}
 Let $\mathcal{X}$ be a  $k$-cube complex with $N$ vertices covered by a cartesian product of $k$ trees with valencies $n_1,n_2,\dots, n_k\in \mathbb{Z}^+$, respectively, where $k\geq 2$ and $N\geq 1$. Let $\Lambda(\mathcal{X})$ be the associated $k$-graph as in Theorem~\ref{thm:N-vertex k graph}. Then, with the notation of subsection~\ref{subs: higher rank graphs}, we have
 \begin{equation}\label{eq:spectral-radius-from valencies}
     \rho(\Lambda(\mathcal{X}))=(n_1,n_2,\dots,n_k).
 \end{equation}
 \end{proposition}

\begin{proof} The assumption on $\mathcal{X}$ says that there are $n_i$ edges (disregarding orientation) of colour $i$ for each $i=1,\dots ,k$. The graph $\Lambda(\mathcal{X})$ is constructed by assigning two edges in its skeleton, of opposite orientation, for each geometric edge in $\mathcal{}X$. Therefore, if $M_i$ denotes the coordinate matrix of $\Lambda(\mathcal{X})$ in colour $i$, we have that $M_i$ is the same as the adjacency operator in $i$-direction $\mathcal{A}_i$. Thus it is a symmetric matrix with largest positive eigenvalue equal to the valency of the tree in colour $i$. Hence, $\rho(M_i)=n_i$ for each $i=1,\dots,k$, as claimed.
\end{proof}

\begin{remark}
The graph of Proposition~\ref{prop:construct k graphs with arbitrary spectral radii} satisfies $\rho(\Lambda)=(2L_1,\dots,2L_k)$.
\end{remark}

Given a strongly connected $k$-graph $\Lambda$, it was shown in  \cite[Corollary~4.6]{AHLRS2} that $C^*(\Lambda)$  admits KMS states at inverse temperature $\beta=1$ for the one-parameter action $\alpha\colon\mathbb{R} \to \operatorname{Aut} C^*(\Lambda)$, the so-called preferred dynamics, characterised by $\alpha_t(\mathbf{s}_\mu) = e^{it\log\rho(\Lambda)\cdot d(\mu)} \mathbf{s}_\mu$, $t\in \mathbb{R}$, $\mu\in \Lambda$.  Following \cite{LLNSW}, define $\mathcal{S} := \{\rho(\Lambda)^{d(\mu) - d(\nu)} \mid \mu,\nu \in \Lambda\text{ are cycles}\}$ and  let $\lambda := \sup\{s \in \mathcal{S} \mid s < 1\}$. By the main result of \cite{LLNSW}, Theorem~3.1, we have  $\lambda \in (0,1]$ and the von Neumann algebra generated by the image of $C^*(\Lambda)$ in the GNS representation $\pi_\varphi$ corresponding to an extremal KMS state $\varphi$ is the
    injective type~$\mathrm{III}_\lambda$ factor.

    Our application here is motivated by \cite[Example 7.7]{LLNSW}, see also \cite{O} and \cite{Y2}. It consists of producing an infinite family of von Neumann factors $(\pi_\varphi(C^*(\Lambda)))''$ of type $\mathrm{III}_\lambda$ associated to $k$-graphs in this fashion. Recall from \cite[Section 6]{LLNSW} that the group of periods of a strongly connected graph $\Lambda$ is defined as $\mathcal{P}_\Lambda = \mathcal{P}_v^+ - \mathcal{P}_v^+$, where for an arbitrary vertex $v\in \Lambda^0$, $\mathcal{P}_v^+ $ is the subsemigroup $d(v\Lambda v)$ of $\N^k$. Equivalently, $
\mathcal{P}_\Lambda$ is the subgroup of $\Z^k$ determined as $ \{d(\mu) - d(\nu) \mid \mu,\nu\text{ are cycles in $\Lambda$}\}$. As shown in \cite[Theorem 7.3]{LLNSW}, the set $\mathcal{S}$ above is the closure inside the positive real half-line of the set $\{\rho(\Lambda)^g\mid g\in \mathcal{P}_\Lambda \}$.

\begin{corollary}\label{cor:von Neumann factors}
For $k\geq 2$ and any $k$-tuple $(L_1,\dots,L_k)$ of positive integers, let $\Lambda$ be the $k$-graph with two vertices from Proposition~\ref{prop:construct k graphs with arbitrary spectral radii}. There is a $\mathrm{III}_\lambda$ von Neumann factor $(\pi_\varphi(C^*(\Lambda)))''$, where $\pi_\varphi$ is the GNS representation of $C^*(\Lambda)$ corresponding to an extremal KMS$_1$ state $\phi$, and the type is determined as
\[
\lambda=\operatorname{sup}\{(2L_1)^{m_1}(2L_2)^{m_2}\dots (2L_k)^{m_k}\mid (m_1,m_2,\dots,m_k)\in \mathcal{P}_\Lambda\}\cap(0,1].
\]
In particular, if $L_1=\cdots=L_k=L$, then $\lambda=(2L)^{-2}$.
\end{corollary}

\begin{proof} Fix $k\geq 2$ and positive integers $L_1,\dots,L_k$, and let $\Lambda$ be as specified.
Then $\Lambda$ is strongly connected, and we may apply Theorem 3.1 of \cite{LLNSW} to obtain the claimed von Neumann factors.

 The remaining task is to compute $\mathcal{P}_\Lambda$. By our construction of $\Lambda$ it is not hard to see that $\mathcal{P}_\Lambda$ is generated by ${m}\in \Z^k$ where either ${m}_i=2$ for a unique $i\in \{1,\dots,k\}$ while ${m}_l=0$ at $l\neq i$, or ${m}_i={m}_j=1$  for some $i\neq j$ in $\{1,\dots,k\}$  and ${m}_l=0$ for $l\notin \{i,j\}$. If $L_1=L_2=\dots=L_k=L$, then $\rho(\Lambda)^{{m}}=(2L)^{\sum_{i=1}^k m_i}$ with ${m}\in \mathcal{P}_\Lambda$, and because $\sum_{i=1}^k m_i\in 2\Z$, the required type is attained as $\lambda=(2L)^{-2}$.

\end{proof}

It was pointed out in \cite[Remark 7.6]{LLNSW} that the  type of the von Neumann factors arising from extremal KMS$_1$ states depends only on the skeleton of the $k$-graph, and not on its factorisation rules. In our examples in Corollary~\ref{cor:von Neumann factors}, this means that the type of the von Neumann factor depends only on the complex $\tilde{P}$ built up as a 2-cover of the one-vertex complex $(\mathcal{T}_{2L_1}\times \dots\times\mathcal{T}_{2L_k})\backslash(\mathbb{F}_{1}\times \cdots \times \mathbb{F}_{k})$.

\subsection{Spectral theory of $k$-graphs}\label{sec:spectral theory}

Alon and Boppana  prove that asymptotically in families of finite $(q+1)$-regular graphs $X_n$ with diameter tending to $\infty$ the largest absolute value of a non-trivial eigenvalue $\lambda(X_n)$ of the adjacency operator $A_{X_n}$ has limes inferior $
\varliminf_{n \to \infty} \lambda(X_n) \geq 2\sqrt{q}.$

 Now, instead of graphs we may consider cube complexes covered by products of trees $\rm{T}_1\times \dots \times\rm{T}_k$, such that $\rm{T}_i$ has valency $q_i$ and look at adjacency operators $\mathcal{A}_i$ in direction $i$ corresponding to an individual tree $\rm{T}_i$.

\begin{definition}
Let $X$ be a finite $k$-cube complex that  has constant valency $q_i+1$ in all directions $i=1,\dots,k$. Then $X$ is a \textbf{cubical Ramanujan complex}, if for each $i \in \{1, \ldots, k\}$, the eigenvalues $\lambda$ of $\mathcal{A}_i$ either satisfy the equality $\lambda = \pm(q_i+1)$ or the bound
\[
\lambda \leq 2\sqrt{q_i}.
\]
\end{definition}

Each such complex yields a $k$-graph $\Delta$  such that $\rho(\Delta)=(q_1+1, q_2+1,...,q_k+1)$.

There are explicit constructions of Ramanujan cube complexes for several infinite families
in \cite{RSV}. We consider next the complexes from \cite{RSV}, corresponding to congruence subgroups of arithmetic lattices.
We reformulate some results of \cite{RSV} in the light of the present paper.

\begin{theorem}(Cf. \cite[Section 6]{RSV})\label{thm:infinite series complexes same valency}
For $p$ a prime, $l$ a positive integer, and $N\geq 2$, there are infinitely many $k$-cube complexes with $N$ vertices covered by products of $k$ trees, where $k\leq p-1$ and each tree is of valency $p^l+1$, satisfying optimal spectral properties, namely with a spectral gap the interval $[2\sqrt{q}, q+1]$, for $q=p^l$.
\end{theorem}

\begin{proof}
Such $k$-cube complexes were constructed in \cite[Section 6]{RSV}. They correspond to congruence quotients
of arithmetic groups. The number of vertices of such complexes is given by the order of the group $\operatorname{PGL}(2,p^l)$.
\end{proof}

\begin{remark}
There are also non-residually finite complexes which have interesting $k$-graphs, although they do not necessarily exhibit the optimal spectral gap. Such complexes with one vertex were constructed in \cite[Section 5]{RSV}. Applying Lemma~\ref{lem:2-vertex 3-complex}, we get such complexes with $2$ vertices for all values $k\geq 1$.
\end{remark}

Now we extend the notion of the Ramanujan cube complexes to higher-rank graphs.
\begin{definition}\label{def:L regular}
We say that a coordinate matrix of a $k$-graph is $L$-regular for $L\in \N$, if the sum of all row entries is equal to $L$.
\end{definition}

\begin {definition}\label{def:ramanujan k graph definition}
Let $\Lambda$ be a $k$-graph with $L_i$-regular coordinate matrices $M_i$ having positive second eigenvalue $\lambda_i$, for  $i=1,\dots,k$. We say that the $k$-graph $\Lambda$ is Ramanujan if
\[
\lambda_i \leq 2\sqrt{L_i-1}\text{ for all }i=1,\dots, k.
\]

\end{definition}

\begin{theorem}\label{prop: Ramanujan higher rank graphs}
For each $k\geq 2$, there is an infinite family of Ramanujan k-graphs with $N\geq 2$ vertices. More precisely, $N$ is determined as the index of  congruence subgroups of the RSV-groups $\Gamma_{M,\delta}$ from subsection~\ref{subsection:RSV groups}.
\end{theorem}

\begin{proof}
This is a direct application of Theorem~\ref{thm:infinite series complexes same valency} in conjunction with Theorem~\ref{thm:N-vertex k graph}.  For a complex $\mathcal{X}$ with $N$ vertices, there is a $k$-graph $\Lambda(\mathcal{X})$ with $N$ vertices by an application of Theorem~\ref{thm:N-vertex k graph}.
\end{proof}

We note that both one-vertex cube complexes covered by product of $k$ trees and
one-vertex higher rank graphs are trivially Ramanujan, so we will require in addition the number of vertices to be greater than, for example, the maximum of $(L_i-1)^2, i=1,...,k$.

\begin{example}\label{ex:Ramanujan from order 25}
We now describe an explicit Ramanujan $3$-graph with  $25$ vertices in the above infinite family.
Let $p=5$ and consider the group $\Gamma_1$ from section \ref{subsection:RSV groups} acting on a product of three trees with valencies $(6,6,6)$. Let $P$  denote the one-vertex $3$-complex associated to $G$.
The existence of the claimed $3$-graph is assured by Proposition~\ref{prop:N vertex complex from symmetric group} and Theorem~\ref{thm:N-vertex k graph} because $\Gamma_1$ has a quotient $L$ of order $25$ (indeed, it has quotients of order $5^l$ for all $l\geq 1$). Let $\mathcal{X}$ denote the resulting complex with $25$ vertices, and let $\Lambda$ be its associated $3$-graph.

Certain subsets of generators of $\Gamma_1$ already generate a group of order $25$ in the cover, as may be verified using MAGMA. For example, the image $Q(a_1)$ in $S_{25}$ is the product of disjoint cycles
\[
Q(a_1)=(1,15,24,8,17)(2,11,25,9,18)(3,12,21,10,19)(4,13,22,6,20)(5,14,23,7,16).
\]
With the notation of Proposition~\ref{prop:N vertex complex from symmetric group}, we have isomorphisms of groups
\begin{align*}
L&\cong \langle Q(a_1), Q(a_5), Q(a_9) \rangle\\
L&\cong \langle Q(b_2), Q(b_6), Q(b_{10})\rangle\\
L&\cong \langle Q(c_3), Q(c_7), Q(c_{11})\rangle,
\end{align*}
thus all three groups in the right hand side are abstractly isomorphic to a {finite} group of order $25$. Let $K_1$, $K_2$ and $K_3$, respectively, denote the Cayley graphs of the finite group of order $25$ coming from the three preceding isomorphisms.

This means that while the presentation of the infinite group $\Gamma_1$  requires generators of all three colours $a_i, b_j, c_k$ (as $\Gamma_1$ is irreducible), in the presentation of the finite group of order $25$, generators of only one colour suffice. The finite cover is the complex $\mathcal{X}$, and fixing each colour yields the Cayley graph of a finite group of order $25$. In other words, each of the generating sets $A_1$, $A_2$ and $A_3$ give Cayley graphs in three different sets of generators (colours) of the \emph{same} finite group.

 The adjacency matrices $M_i$ of the Cayley graphs $K_i$, $1=1,2,3$, may be computed using MAGMA, using that $Q(G)$ acts as permutations in $S_{25}$. It turns out that $M_1, M_2, M_3$ are equal. As noted in the proof of Proposition~\ref{prop:spectral radii}, the adjacency matrices $M_1,M_2, M_3$ of the complex are also the adjacency matrices of the $3$-graph $\Lambda$. Each $M_i$ is $6$-regular in the sense of Definition~\ref{def:L regular}, for $i=1,2,3$, as may be seen from the concrete description of the matrices obtained with MAGMA. An application of Theorem~\ref{prop: Ramanujan higher rank graphs}
 gives  that $\Lambda$ is a Ramanujan
 $3$-graph, so the second largest eigenvalue $\lambda_i$ of $M_i$ is dominated by $2\sqrt{5}$, for $i=1,2,3$.

  In this example, the spectral gap is strictly in the optimal bound, namely, the second eigenvalue of $M_1$ is dominated by $3.24$, according to MAGMA computations. This bound is lower than the theoretically predicted $2\sqrt{5}$.

Also using MAGMA shows that the product $M_1M_2M_3$ is not a $(0,1)$-matrix, which distinguishes this example from
\cite{RS} and all papers inspired by it. For example, the diagonal entries in  $M_1M_2M_3$ are all equal to $12$. The remaining entries are $6,7$ or $15$.

\section{Matrices of the Cayley graph of an order $25$-group}
\[
 \begin{bmatrix}
0 &1 &1 &1 &1 &1 &1 &0 &0 &0 &0 &0 &0 &0 &0 &0 &0 &0 &0 &0 &0 &0 &0 &0 &0\\
1 &0 &1 &0 &0 &0 &1 &1 &1 &1 &0 &0 &0 &0 &0 &0 &0 &0 &0 &0 &0 &0 &0 &0 &0\\
1 &1 &0 &1 &0 &0 &0 &0 &1 &0 &1 &1 &0 &0 &0 &0 &0 &0 &0 &0 &0 &0 &0 &0 &0\\
1 &0 &1 &0 &1 &0 &0 &0 &0 &0 &0 &1 &1 &1 &0 &0 &0 &0 &0 &0 &0 &0 &0 &0 &0\\
1 &0 &0 &1 &0 &1 &0 &0 &0 &0 &0 &0 &0 &1 &1 &1 &0 &0 &0 &0 &0 &0 &0 &0 &0\\
1 &0 &0 &0 &1 &0 &1 &0 &0 &0 &0 &0 &0 &0 &0 &1 &1 &1 &0 &0 &0 &0 &0 &0 &0\\
1 &1 &0 &0 &0 &1 &0 &0 &0 &1 &0 &0 &0 &0 &0 &0 &0 &1 &1 &0 &0 &0 &0 &0 &0\\
0 &1 &0 &0 &0 &0 &0 &0 &1 &1 &0 &0 &0 &0 &1 &0 &0 &0 &0 &1 &1 &0 &0 &0 &0\\
0 &1 &1 &0 &0 &0 &0 &1 &0 &0 &1 &0 &0 &0 &0 &0 &0 &0 &0 &1 &0 &1 &0 &0 &0\\
0 &1 &0 &0 &0 &0 &1 &1 &0 &0 &0 &0 &0 &0 &0 &0 &0 &0 &1 &0 &1 &0 &1 &0 &0\\
0 &0 &1 &0 &0 &0 &0 &0 &1 &0 &0 &1 &0 &0 &0 &0 &1 &0 &0 &0 &0 &1 &0 &1 &0\\
0 &0 &1 &1 &0 &0 &0 &0 &0 &0 &1 &0 &1 &0 &0 &0 &0 &0 &0 &0 &0 &0 &1 &1 &0\\
0 &0 &0 &1 &0 &0 &0 &0 &0 &0 &0 &1 &0 &1 &0 &0 &0 &0 &1 &0 &0 &0 &1 &0 &1\\
0 &0 &0 &1 &1 &0 &0 &0 &0 &0 &0 &0 &1 &0 &1 &0 &0 &0 &0 &1 &0 &0 &0 &0 &1\\
0 &0 &0 &0 &1 &0 &0 &1 &0 &0 &0 &0 &0 &1 &0 &1 &0 &0 &0 &1 &1 &0 &0 &0 &0\\
0 &0 &0 &0 &1 &1 &0 &0 &0 &0 &0 &0 &0 &0 &1 &0 &1 &0 &0 &0 &1 &0 &0 &1 &0\\
0 &0 &0 &0 &0 &1 &0 &0 &0 &0 &1 &0 &0 &0 &0 &1 &0 &1 &0 &0 &0 &1 &0 &1 &0\\
0 &0 &0 &0 &0 &1 &1 &0 &0 &0 &0 &0 &0 &0 &0 &0 &1 &0 &1 &0 &0 &1 &0 &0 &1\\
0 &0 &0 &0 &0 &0 &1 &0 &0 &1 &0 &0 &1 &0 &0 &0 &0 &1 &0 &0 &0 &0 &1 &0 &1\\
0 &0 &0 &0 &0 &0 &0 &1 &1 &0 &0 &0 &0 &1 &1 &0 &0 &0 &0 &0 &0 &1 &0 &0 &1\\
0 &0 &0 &0 &0 &0 &0 &1 &0 &1 &0 &0 &0 &0 &1 &1 &0 &0 &0 &0 &0 &0 &1 &1 &0\\
0 &0 &0 &0 &0 &0 &0 &0 &1 &0 &1 &0 &0 &0 &0 &0 &1 &1 &0 &1 &0 &0 &0 &0 &1\\
0 &0 &0 &0 &0 &0 &0 &0 &0 &1 &0 &1 &1 &0 &0 &0 &0 &0 &1 &0 &1 &0 &0 &1 &0\\
0 &0 &0 &0 &0 &0 &0 &0 &0 &0 &1 &1 &0 &0 &0 &1 &1 &0 &0 &0 &1 &0 &1 &0 &]\\
0 &0 &0 &0 &0 &0 &0 &0 &0 &0 &0 &0 &1 &1 &0 &0 &0 &1 &1 &1 &0 &1 &0 &0 &0\\
\end{bmatrix}
\]

\end{example}


\begin{thebibliography}{99}

\bibitem{Ballmann-Brin1995} W.~Ballmann, M.~Brin, {\em Orbihedra of nonpositive curvature}, Inst. Hautes \'Etudes Sci. Publ. Math. No. 82 (1995), 169-209.



\bibitem{BS72}
E.~Brieskorn and K.~Saito, \emph{Artin-{G}ruppen und {C}oxeter-{G}ruppen}, Invent. Math., {\bf 17}(1972), 245--271.

\bibitem{burger-mozes:lattices}
M.~Burger, Sh.~Mozes,
{\em Lattices in product of trees},
Inst.\ Hautes \'Etudes Sci.\ Publ.\ Math. \textbf{92} (2000), 151--194.

\bibitem{CK} J.~Cuntz and W.~Krieger, \emph{A class of $C^*$-algebras and topological Markov chains}, Invent. Math. {\bf 56} (1980), 251--268.


\bibitem{Garside-book}
P.~Dehornoy, F.~Digne, E.~Godelle, D.~Krammer and J.~Michel, Foundations of Garside Theory, Number 22 in EMS Tracts in Mathematics, European Math. Soc., 2015.

\bibitem{E} D.G.~Evans, \emph{On the K-theory of higher rank graph $C^*$-algebras}, New York J. Math. {\bf 14} (2008), 1--31.


\bibitem{FS} N.~Fowler and A.~Sims, \emph{Product systems over right-angled Artin semigroups}, Trans. Amer. Math. Soc. {\bf 354} (2002), 1487--1509.

\bibitem{Gromov} M.~Gromov, Hyperbolic groups, Essays in Group Theory (S.M. Gersten, ed.), {\bf 8}, Math. Sci. Res. Inst. Publ., (1988) , 75–263.

\bibitem{Hatcher} A. Hatcher, Algebraic Topology, electronic version available from author's webpage.

\bibitem{HRSW} R. Hazlewood, I. Raeburn, A. Sims and S.B.G. Webster \emph{Remarks on some
    fundamental results about higher-rank graphs and their $C^*$-algebras},
    Proc. Edinb. Math. Soc. (2) \textbf{56} (2013), no.~2, 575--597.

\bibitem{AHLRS2} A. an Huef, M. Laca, I. Raeburn and A. Sims, \emph{KMS states on the
    C$^*$-algebra of a higher-rank graph and periodicity in the path space}, J. Funct. Anal. {\bf 268} (2015), no. 7, 1840–-1875.

\bibitem{KR} J.S.~Kimberley and G.~Robertson, \emph{Groups acting on products of trees, tiling systems and analytic K-theory}, New York J. Math. {\bf 8} (2002), 111–131.


 \bibitem{KP} A.~Kumjian and D.~Pask, \emph{Higher rank graph $C^*$-algebras} New York J. Math. {\bf 6} (2000), 1–20.

\bibitem{LLNSW}M.~Laca, N.S.~ Larsen, S.~Neshveyev, A.~Sims, S.B.G.~Webster, {\em Von Neumann algebras of strongly connected higher-rank graphs},
Math. Ann. {\bf 363} (2015), no. 1-2, 657--678.

\bibitem{LV} M.~V.~Lawson, A.~Vdovina, {\em Higher dimensional generalizations of the Thompson groups},
Adv. in Math. 369 (2020), 107191.

\bibitem{LSV} M.~V.~Lawson, A.~Sims and A.~Vdovina,  {\em Higher dimensional generalizations of the Thompson groups via higher rank graphs}, preprint, arXiv:2010.08960v1[math.GR].

\bibitem{MacL} S. Mac Lane, Categories for the working mathematician, Second edition, 1978 Springer-Verlag New York.

\bibitem{M} S.A.~Mutter, \emph{K-theory of 2-rank graphs associated to complete bipartite graphs}, preprint, arXiv:2004:11602v1[math.CO].

\bibitem{MRV} S.A. Mutter, A.C.~Radu and A.~Vdovina, \emph{C*-algebras of higher-rank graphs from groups acting in buildings, and explicit computation of their K-theory}, preprint, arXiv:2012.05561v1[math.OA].

\bibitem{O} R. Okayasu, \emph{Type III factors arising from Cuntz-Krieger algebras},
    Proc. Amer. Math. Soc. {\bf 131} (2003), no.~7, 2145--2153.

\bibitem{Ore} O. Ore, Theory of graphs, Amer. Math. Soc. Colloquium Publications, vol XXXVIII, R. I. 1962.

\bibitem{Phi} N.C.~Phillips, \emph{A classification theorem for nuclear purely infinite simple $C^*$-algebras}, Doc. Math. {\bf 5} (2000), 49--114.

\bibitem{Pow} S.C.~Power,\emph{Classifying higher rank analytic Toeplitz algebras}, New York J. Math. {\bf 13} (2007), 271--298.

 \bibitem{RSY}I.~Raeburn, A.~Sims and T.~Yeend, Higher-rank graphs and their $C^*$-algebras, Proc. Edinburgh Math. Soc. 46 (2003), 99--115.

 \bibitem{rattaggi:thesis} D.Rattaggi,
{\em Computations in groups acting on a product of trees: Normal subgroup structures and quaternion lattices}.  Thesis (Dr.sc.math.) ETH Z\"urich (Switzerland), 2004, 305 pages, ProQuestLLC.


 \bibitem{RSV}N.~Rungtanapirom, J.~Stix, A.~Vdovina, {\em  Infinite series of quaternionic 1-vertex cube complexes, the doubling construction, and explicit cubical Ramanujan complexes.}
International Journal of Algebra Computation 29 (2019), no. 6, 951--1007.


\bibitem{R} G.~Robertson, \emph{Boundary $C^*$-algebras of triangle geometries}, J. Funct. Anal. 266 (2014), no. 11, 6319–6334.

\bibitem{RS} G.~Robertson and T.~Steger, \emph{Affine buildings, tiling systems and higher rank Cuntz-Krieger algebras}, J. Reine Angew. Math. {\bf 513} (1999), 115–144.

  \bibitem{RS2} G.~Robertson and T.~Steger, \emph{Asymptotic K-theory for groups acting on  $\tilde{A}_2$ buildings}, Canad. J. Math. 53 (2001), no. 4, 809–833.


\bibitem{sageev}
M.~Sageev,
CAT(0) cube complexes and groups,
in: \textit{Geometric Group Theory}, editors: M.~Bestvina, M.~Sageev, K.~Vogtmann, IAS/Park City Mathematics Series Volume \textbf{21}, 2014, 399pp.

\bibitem{Saito}
K. Saito, {\em Growth functions for Artin monoids}, Proc. Japan Acad., {\bf 85}, Ser. A (2009), 84--88.

\bibitem{Sim} A.~Sims, \emph{Gauge-invariant ideals in the $C^*$-algebras of finitely aligned higher rank graphs}, Canad. J. Math. {\bf 58} (2006), no. 6, 1268-1290.

\bibitem{Spi} J.~Spielberg, \emph{Free product groups, Cuntz-Krieger algebras, and covariant maps}, Internat. J. Math. {\bf 2}(1991), 457--476.


\bibitem{vdovina-YB} A. Vdovina, {\em Drinfeld-Manin solutions of the Yang-Baxter equation coming from cube complexes}, Internat. J. Algebra Comput. {\bf 31} (2021), no. 4, 775--788.

\bibitem{vdovina-pol} A.~Vdovina, \emph{Combinatorial structure of some hyperbolic buildings}, Math. Z. {\bf 241} (2002), 471--478.

\bibitem{wise1} D. Wise, {\em Non-positively curved squared complexes: Aperiodic tilings and non-residually finite groups.} Thesis (Ph.D.)Princeton University. 1996

\bibitem{Y} D. Yang, {\em The interplay between $k$-graphs and the Yang-Baxter equation}, J. Algebra {\bf 451} (2016), 494-525.

\bibitem{Y2} D. Yang, \emph{Factoriality and type classification of $k$-graph von Neumann algebras},  Proc. Edinb. Math. Soc. {\bf 60} (2017), no. 2, 499-–518.
\end{thebibliography}
\end{document}